\documentclass[9pt,shortpaper,twoside,web]{ieeecolor}
\usepackage{amsmath,amssymb,amsfonts}
\usepackage{generic}
\usepackage{enumerate}
\usepackage{graphicx}
\usepackage{cite}
\usepackage{mathtools}

\newtheorem{theorem}{Theorem}
\newtheorem{definition}{Definition}
\newtheorem{proposition}{Proposition}
\newtheorem{lemma}{Lemma}
\newtheorem{corollary}{Corollary}
\newtheorem{example}{Example}
\newtheorem{problem}{Problem}
\newtheorem{remark}{Remark}

\begin{document}

\title{Quotients of Probabilistic Boolean Networks}

\author{Rui~Li, Qi~Zhang, and Tianguang~Chu%
\thanks{R. Li is with the School of Mathematical Sciences, Dalian University
of Technology, Dalian 116024, China (e-mail: rui\_li@dlut.edu.cn).}%
\thanks{Q. Zhang is with the School of Information Technology and Management,
University of International Business and Economics, Beijing 100029,
China (zhangqi@uibe.edu.cn).}%
\thanks{T. Chu is with the College of Engineering, Peking University,
Beijing 100871, China (e-mail: chutg@pku.edu.cn).}}

\maketitle

\begin{abstract}
A probabilistic Boolean network (PBN) is a discrete-time system
composed of a collection of Boolean networks between which the PBN
switches in a stochastic manner. This paper focuses on the study of
quotients of PBNs. Given a PBN and an equivalence relation on its
state set, we consider a probabilistic transition system that is
generated by the PBN; the resulting quotient transition system then
automatically captures the quotient behavior of this PBN. We
therefore describe a method for obtaining a probabilistic Boolean
system that generates the transitions of the quotient transition
system. Applications of this quotient description are discussed, and
it is shown that for PBNs, controller synthesis can be performed
easily by first controlling a quotient system and then lifting the
control law back to the original network. A biological example is
given to show the usefulness of the developed results.
\end{abstract}


\IEEEpeerreviewmaketitle

\section{Introduction}

\label{sec1}

Mathematical modeling of biological systems is a valuable avenue for
understanding complex biological systems and their behaviors. One
powerful approach to modeling biological systems is through a
Boolean model, where each system component is characterized with a
binary variable. Boolean network (BN) modeling can capture the
system's behavior without the need for much kinetic detail, making
it a practical choice for systems where enough kinetic information
may not be at disposal \cite{albert2014}. A BN is typically placed
in the form of a (deterministic) nonlinear system (with a finite
state space); while interestingly, based on an algebraic state
representation approach, the Boolean dynamics can be exactly mapped
into the standard discrete-time linear dynamics \cite{chengbook}.
This formal simplicity makes it relatively easy to formulate and
solve classical control-theoretic problems for BNs, and thereby has
stimulated a great many interesting subsequent developments in this
area
\cite{laschov2014,bof2015,fornasini2015,lu2016b,zhang2016c,guo2017,li2017b,liang2017b,liu2017,cheng2018,rafimanzelat2019,wang2019b,yu2019,zhang2020b,zhang2020,li2020c,li2020d,zhong2020}.
For some recent work on the analysis and control of BNs based on
other approaches, see, e.g., \cite{weiss2019,chen2020b,zhang2020c}.

A probabilistic Boolean network (PBN) is a stochastic extension of
the classical BN. It can be considered as a collection of BNs
endowed with a probability structure describing the likelihood with
which a constituent network is active. PBNs possess not only the
appealing properties of BNs such as requiring few kinetic
parameters, but also are able to cope with uncertainties, both in
the experimental data and in the model selection
\cite{shmulevich2002a}. The algebraic state representation has also
proved a powerful framework for studying control-related problems in
PBNs. Examples of recent studies based on the algebraic
representation approach include investigations of network robustness
and synchronization \cite{li2020b,zhu2020,chen2018}, controllability
and stabilizability
\cite{meng2019b,guo2019,liu2019,zhou2020,huang2020}, observability
and detectability \cite{zhou2019,fornasini2020,wang2019}, optimal
control \cite{wu2020}, just to quote a few.

It is a well-known fact that the analysis of control systems and
synthesis of controllers become increasingly difficult as the
dimension of the system gets larger. It is then desirable to have a
methodology that reduces the size of control systems while
preserving the properties relevant for analysis or synthesis.
Quotient systems can be seen as lower dimensional models that may
still contain enough information about the original system. A
stability analysis of BNs based on a quotient map was presented in
\cite{wang2014} and \cite{guo2018c}, where it was shown that the
stability of the original BN can be inferred from the analysis of a
specific quotient dynamics. Our recent work described a process for
obtaining quotients of BNs \cite{r.li2020}. A relation-based
transformation strategy was introduced, which is able to transform a
BN expressed in algebraic form into a quotient Boolean system suited
for use. The present paper focuses on the study of quotients of
PBNs. Given a PBN, together with an equivalence relation on the
state set, we consider a probabilistic transition system
$\mathcal{T}$ that is generated by the PBN. The equivalence relation
then naturally induces a partition of the state space of
$\mathcal{T}$, and the corresponding quotient system fully captures
the quotient dynamics of the PBN concerned. We therefore develop a
probabilistic Boolean system that produces the transitions of the
quotient transition system. As an application of this quotient
description, we apply the proposed technique to solve two typical
control problems, namely the stabilization and optimal control
problems. The results show us that through the use of an
appropriately defined relation, the proposed quotient system can
indeed preserve the system property relevant to control design.
Consequently, synthesizing controllers for a PBN can be done easily
by first designing control polices on the quotient and then inducing
the control polices back to the original network.

The remainder of this paper is organized as follows.
Section~\ref{sec2} contains the basic notation and briefly reviews
PBNs and probabilistic transition systems. Section~\ref{sec3}
details a process for generating quotients of PBNs given that the
networks are represented in algebraic form. Section~\ref{sec4}
discusses the use of the proposed quotient systems for control
design and presents applications to stabilization and optimal
control problems. Section~\ref{sec5} gives a biological example
illustrating the developed results. A summary of the paper is given
in the last section.

\section{Notation and Preliminaries}

\label{sec2}

\subsection{Notation}

\label{sec2.1}

The following notation is used throughout the paper. The symbol
$\delta_k^i$ denotes the $i$th $k \times 1$ canonical basis vector
(all entries of $\delta_k^i$ are $0$ except for the $i$th one, which
is $1$), $\Delta_k$ denotes the set consisting of the canonical
vectors $\delta_k^1, \ldots, \delta_k^k$, and $\mathcal{L}^{k \times
r}$ denotes the set of all $k \times r$ matrices whose columns are
canonical basis vectors of length $k$. Elements of $\mathcal{L}^{k
\times r}$ are called logical matrices (of size $k \times r$). A
$(0,1)$-matrix is a matrix with all entries either $0$ or $1$. The
$(i,j)$-entry of a matrix $A$ is denoted by $(A)_{ij}$. Given two
$(0,1)$-matrices $A$ and $B$ of the same size, by $A \leq B$ we mean
that if $(A)_{ij} = 1$ then $(B)_{ij} = 1$ for every $i$ and $j$.
The meet of $A$ and $B$, denoted by $A \wedge B$, is the
$(0,1)$-matrix whose $(i,j)$-entry is $(A)_{ij} \wedge (B)_{ij}$.
The (\textit{left}) \textit{semitensor product} \cite{chengbook} of
two matrices $C$ and $D$ of sizes $k_1 \times r_1$ and $k_2 \times
r_2$, respectively, denoted by $C \ltimes D$, is defined by $C
\ltimes D = (C \otimes I_{l/ r_1})(D \otimes I_{l/k_2})$, where
$\otimes$ is the Kronecker product of matrices, and $I_{l/ r_1}$ and
$I_{l/k_2}$ are the identity matrices of orders $l/r_1$ and $l/k_2$,
respectively, with $l$ being the least common multiple of $r_1$ and
$k_2$.

\subsection{Probabilistic Boolean Networks}

\label{sec2.2}

A PBN is described by the following stochastic equation
\begin{equation} \label{eq1}
X(t+1) = f_{\theta (t)} (X(t), U(t)) ,
\end{equation}
where $X(t) = [X_1(t), \ldots, X_n(t)]^{\top} \in \{1,0\}^n$ is the
state, $U(t) = [U_1(t), \ldots, U_m(t)]^{\top} \in \{1,0\}^m$ is the
control, $\{\theta(t) \colon t = 0,1,\ldots\}$ is a stochastic
process consisting of independent and identically distributed
(i.i.d.) random variables taking values in a finite set $\mathbb{S}
= \{1,\ldots, S\}$, and $f_i$ ($i = 1, \ldots, S$) are Boolean
functions from $\{1,0\}^{n+m}$ to $\{1,0\}^n$. By performing a
matrix expression of Boolean logic and using the semitensor product,
model (\ref{eq1}) can be cast in a form similar to a random jump
linear system with i.i.d. jumps. To be more precise, we let $x(t) =
x_1(t) \ltimes \cdots \ltimes x_n(t)$ and $u(t) = u_1(t) \ltimes
\cdots \ltimes u_m(t)$, where $x_i(t) = [X_i(t), \neg X_i(t)]^\top$
and $u_j(t) = [U_j(t), \neg U_j(t)]^\top$. Then it is shown that the
PBN (\ref{eq1}) satisfies the following algebraic description
\begin{equation*}
x(t+1) = F_{\theta(t)} \ltimes u(t) \ltimes x(t) ,
\end{equation*}
where $x(t) \in \Delta_N$, $u(t) \in \Delta_M$, and $F_i \in
\mathcal{L}^{N \times NM}$ for $i = 1, \ldots, S$, with $N \coloneqq
2^n$ and $M \coloneqq 2^m$. For more information about obtaining the
algebraic description, as well as the properties of the semitensor
product, the reader is referred to, e.g., the monograph of Cheng
\textit{et al.} \cite{chengbook}.

\subsection{Probabilistic Transition Systems}

\label{sec2.3}

Our discussion of quotients of PBNs will draw on the notion of
probabilistic transition systems. Recall that a probability
distribution over a finite set $Q$ is a function $\mu \colon Q
\rightarrow [0,1]$ such that $\sum_{q \in Q} \mu(q) = 1$. The set of
all probability distributions over $Q$ is denoted by
$\text{Dist}(Q)$. We state the following definition.

\begin{definition} [see, e.g., \cite{baier2000,hermanns2011}] \label{def1}
A~\textit{probabilistic transition system} (or \textit{probabilistic
automaton}) is a tuple $\mathcal{T} = (Q, Act, \rightarrow)$, where
$Q$ is a finite set of states, $Act$ is a finite set of actions, and
$\rightarrow \, \subseteq Q \times Act \times \text{Dist}(Q)$ is a
probabilistic transition relation.
\end{definition}

Intuitively, a transition $(q, \alpha, \mu) \in \rightarrow$ means
that in the state $q$ an action $\alpha$ can be executed after which
the probability to move to a state $q' \in Q$ is $\mu (q')$.
Following standard conventions we denote $q \xrightarrow{\alpha}
\mu$ if $(q, \alpha, \mu) \in \rightarrow$. A probabilistic
transition system is \textit{reactive}\footnote{We note that some
authors use the terminology ``reactive" for a probabilistic
transition system where there is at most one (but perhaps no)
transition on a given action from a given state.} if for any state
$q \in Q$ and any action $\alpha \in Act$ there exists a unique $\mu
\in \text{Dist}(Q)$ such that $q \xrightarrow{\alpha} \mu$
\cite{feng2014}. As we will explain in the following section, every
PBN corresponds naturally to a probabilistic transition system which
is always reactive.

Recall that an equivalence relation $\mathcal{R}$ on $Q$ is a
reflexive, symmetric, and transitive binary relation on $Q$. Let $Q
/ \mathcal{R}$ be the quotient set of $Q$ by $\mathcal{R}$ (i.e.,
the set of all equivalence classes $[q] = \{p \in Q \colon (q,p) \in
\mathcal{R}\}$ for $q \in Q$). Then every $\mu \in \text{Dist}(Q)$
induces a probability distribution $\bar{\mu}$ over $Q /
\mathcal{R}$ given by $\bar{\mu} ([q]) = \sum_{p \in [q]} \mu (p)$.
The following definition of a quotient transition system is taken
from \cite[Definition 12]{zhang2016d}, but slightly adjusted to our
notation.

\begin{definition} \label{def2}
Let $\mathcal{T} = (Q, Act, \rightarrow)$ be a probabilistic
transition system and let $\mathcal{R}$ be an equivalence relation
on $Q$. The \textit{quotient transition system}
$\mathcal{T}/\mathcal{R}$ is defined by $\mathcal{T}/\mathcal{R} =
(Q/ \mathcal{R}, Act, \rightarrow_{\mathcal{R}})$, where the
probabilistic transition relation $\rightarrow_{\mathcal{R}}$ is
defined as follows: for any $[q] \in Q/ \mathcal{R}$ and $\pi \in
\text{Dist}(Q/ \mathcal{R})$, $[q]
\xrightarrow{\alpha}_{\mathcal{R}} \pi$ if and only if for every $p
\in [q]$ there exists a $\mu \in \text{Dist}(Q)$ inducing $\pi$ such
that $p \xrightarrow{\alpha} \mu$.
\end{definition}

It follows from the above definition that an action $\alpha$ can be
executed in $[q]$ just in case: (i) $\alpha$ can be executed in
every state in $[q]$, and (ii) all states in $[q]$ have identical
transition probabilities to each of the equivalence classes after
the action $\alpha$. Furthermore, the transition probability in
$\mathcal{T}/{\mathcal{R}}$ from $[q]$ to $[q']$ is simply the
probability with which $\mathcal{T}$ transitions from $q$ (or any
other state belonging to $[q]$) to the equivalence class $[q']$.
Note that $\mathcal{T}/ \mathcal{R}$ may not be reactive even if
$\mathcal{T}$ is. Indeed, it is possible that there are two states
in a class, say $[q]$, which have different probabilities of
transitioning to some equivalence class under a given action, say
$\alpha$, thus violating the above condition (ii). Then the action
$\alpha$ is not executable in $[q]$ and, consequently, the quotient
transition system $\mathcal{T}/ \mathcal{R}$ is not reactive.

In the next section, we will use a similar framework to study
quotients of a PBN.

\section{Construction of Quotients}

\label{sec3}

Let us consider a PBN described by\footnote{Here $N$ and $M$ are in
fact certain powers of $2$, but we do not need this fact in our
argument.}
\begin{equation} \label{eq2}
\Sigma \colon \;\, x(t+1) = F_{\theta(t)} \ltimes u(t) \ltimes x(t)
, \;\; x \in \Delta_N , \;\; u \in \Delta_M .
\end{equation}
As assumed above, $\{\theta(t)\}$ is an i.i.d. process taking
finitely many values $1, \ldots, S$ with associated probabilities
$p_1, \ldots, p_S$; and $F_i \in \mathcal{L}^{N \times NM}$ for each
$1 \leq i \leq S$. We define a column-stochastic matrix\footnote{A
matrix is \textit{column-stochastic} if all entries are nonnegative
and each column sums to one.} $P = p_1 F_1 + p_2 F_2 + \cdots + p_S
F_S $, and for each $u \in \Delta_M$ let
\begin{equation} \label{eq2a}
P(u) = P \ltimes u .
\end{equation}
The $(i,j)$-entry of $P(u)$ then gives the transition probability of
$\Sigma$ from its state $\delta_N^j$ to state $\delta_N^i$ when
input $u$ is applied (see, e.g., \cite{chengbook}). The above matrix
$P$ is called the \textit{transition probability matrix} of $\Sigma$
\cite{zhou2020}. Note that any column-stochastic matrix $P$ of size
$N \times NM$ can be interpreted as the transition probability
matrix of a PBN of the form (\ref{eq2}). Indeed, since every
column-stochastic matrix is a convex combination of logical matrices
(cf. the algorithms in \cite{ching2009} and \cite{kobayashi2017}),
there exist logical matrices $F_1, \ldots, F_S$ and positive reals
$\lambda_1, \ldots, \lambda_S$ such that $P = \sum_{i=1}^S \lambda_i
F_i$ and $\sum_{i=1}^S \lambda_i = 1$. Let $\{\theta(t)\}$ be the
i.i.d. process with the probability that $\theta(t) = i$ equal to
$\lambda_i$ for all $t \geq 0$. Then the PBN described in
(\ref{eq2}) has as its transition probability matrix the matrix $P$.

In order to investigate quotients of (\ref{eq2}), we first recall
that every equivalence relation $\mathcal{R} \subseteq \Delta_N
\times \Delta_N$ can be viewed as induced by a logical matrix $C$
with $N$ columns and full row rank, by saying
\begin{equation} \label{eq3}
(x, x') \in \mathcal{R} \Longleftrightarrow C x = C x' .
\end{equation}
The matrix $C$ is easily derived from the matrix representation of
$\mathcal{R}$. Indeed, let $A_{\mathcal{R}}$ be the $N \times N$
matrix with entries
\begin{equation*}
(A_{\mathcal{R}})_{ij} =
\begin{cases}
1 & \text{if $(\delta_N^i , \delta_N^j) \in
\mathcal{R}$}, \\
0 & \text{otherwise}.
\end{cases}
\end{equation*}
If $C$ is a matrix having the same set of distinct rows as
$A_{\mathcal{R}}$, but with no rows repeated, then it must be a
logical matrix of full row rank and fulfilling condition (\ref{eq3})
(see \cite[Lemma 4.6]{r.li2021} where it is shown that such a $C$ is
a logical matrix with no zero rows, hence of full row rank, and
(\ref{eq3}) holds for that $C$). Note that, for an equivalence
relation $\mathcal{R} \subseteq \Delta_N \times \Delta_N$ induced by
a matrix $C \in \mathcal{L}^{\widetilde{N} \times N}$ of full row
rank, the quotient set $\Delta_N / \mathcal{R}$ has cardinality
$\widetilde{N}$, and the correspondence $[x] \mapsto Cx$ gives a
bijection between $\Delta_N / \mathcal{R}$ and
$\Delta_{\widetilde{N}}$.

We now consider quotients of (\ref{eq2}). The PBN (\ref{eq2})
naturally generates a probabilistic transition system
$\mathcal{T}(\Sigma) = (\Delta_N, \Delta_M, \rightarrow)$, where the
transition relation $\rightarrow$ is defined as follows: for $a \in
\Delta_N$, $u \in \Delta_M$, and $\mu \in \text{Dist}(\Delta_N)$,
\begin{equation*}
a \xrightarrow{u} \mu \Longleftrightarrow \text{$\mu (x) = x^{\top}
P(u) a$ for all $x \in \Delta_N$} .
\end{equation*}
Here, $x^{\top} P(u) a$ is just the transition probability of
$\Sigma$ moving from $a$ to $x$ under input $u$, since it coincides
with the $(i,j)$-entry of $P(u)$ when $x = \delta_N^i$ and $a =
\delta_N^j$. The above definition of $\rightarrow$ then says that,
for each state $a \in \Delta_N$ and any $u \in \Delta_M$, the
probability of $\mathcal{T}(\Sigma)$ transitioning to the next state
$x$ is exactly the same as the probability of $\Sigma$ transitioning
from $a$ to $x$. Clearly, the transition system
$\mathcal{T}(\Sigma)$ generated in this way is reactive. In view of
the following discussion, we mention that the converse of this fact
is also true. Indeed, given a reactive transition system
$\mathcal{T}' = (\Delta_N, \Delta_M, \rightarrow ')$, for each $u
\in \Delta_M$ define $P'(u)$ to be the $N \times N$ matrix with
$(i,j)$-entry $(P' (u))_{ij} = \mu (\delta_N^i)$, where $\mu$ is the
unique probability distribution on $\Delta_N$ such that $\delta_N^j
\mathrel{{\xrightarrow{u}}{}'} \mu$. Set $P' =$ $\left[
P'(\delta_M^1) \; \cdots \; P'(\delta_M^M) \right]$. Then $P'$ is
column-stochastic (since each $P'(u)$ is), and the system
$\mathcal{T}'$ can be considered as generated by a PBN whose
transition probability matrix is $P'$.

Let $\mathcal{R}$ be an equivalence relation on $\Delta_N$ and
consider the quotient transition system $\mathcal{T}(\Sigma)/
\mathcal{R} = (\Delta_N / \mathcal{R}, \Delta_M,
\rightarrow_{\mathcal{R}})$. For the analysis to remain in the
Boolean context, we expect that the transitions of
$\mathcal{T}(\Sigma) / \mathcal{R}$ are also generated by a Boolean
system\footnote{In the following, we use the term ``probabilistic
Boolean system" to refer to a stochastic system of the form
(\ref{eq2}) where $N$ and $M$ are not restricted to be powers of
$2$.} of the form (\ref{eq2}). By the above argument, this is the
case exactly when $\mathcal{T}(\Sigma) / \mathcal{R}$ is reactive,
or equivalently, when
\begin{multline} \label{eq4}
\sum_{x \in [b]} x^{\top} P(u) a = \sum_{x \in [b]} x^{\top} P(u) a'
, \;\; \forall u \in \Delta_M, \; \forall [b] \in \Delta_N /
\mathcal{R}, \\
\text{$\forall a, a' \in \Delta_N$ with $(a, a') \in \mathcal{R}$}
\end{multline}
(that is, for any control action, states in the same class have the
same transition probabilities to any equivalence class). We
therefore restrict our attention to those $\mathcal{R}$ satisfying
(\ref{eq4}). The following theorem gives a method for constructing a
probabilistic Boolean system that generates the transitions of
$\mathcal{T}(\Sigma) / \mathcal{R}$.

\begin{theorem} \label{thm1}
Consider a PBN $\Sigma$ as in (\ref{eq2}) and let $P(u)$ be as in
(\ref{eq2a}). Suppose that $\mathcal{R}$ is an equivalence relation
on $\Delta_N$ induced by a matrix $C \in \mathcal{L}^{\widetilde{N}
\times N}$ of full row rank, and that property (\ref{eq4}) holds.
Let $\widetilde{C} \in \mathcal{L}^{N \times \widetilde{N}}$ be such
that\footnote{Since $C$ (being logical) has full row rank, the
transpose $C^{\top}$ does not contain zero columns, so such a
$\widetilde{C}$ must exist.} $\widetilde{C} \leq C^{\top}$, and for
each $u \in \Delta_M$ define $\widetilde{P}(u)$ to be the
$\widetilde{N} \times \widetilde{N}$ matrix given by
$\widetilde{P}(u) = C P(u) \widetilde{C}$. Then:
\begin{enumerate} [(a)]
\item Each $\widetilde{P}(u)$ is column-stochastic.

\item Let
\begin{equation*} \hspace{-1.5em}
\Sigma_{\mathcal{R}} \colon \; x_{\mathcal{R}}(t+1) =
\widetilde{F}_{\tilde{\theta}(t)} \ltimes u(t) \ltimes
x_{\mathcal{R}}(t) , \;\, x_{\mathcal{R}} \in \Delta_{\widetilde{N}}
, \;\, u \in \Delta_M
\end{equation*}
be a probabilistic Boolean system that has $\widetilde{P} =
\big[\widetilde{P}(\delta_M^1)\;\; \widetilde{P}(\delta_M^2)$ $
\cdots \;\, \widetilde{P}(\delta_M^M) \big]$ as its transition
probability matrix. For any $a, a' \in \Delta_N$ and any $u \in
\Delta_M$, the transition probability of $\Sigma_{\mathcal{R}}$ from
$Ca$ to $Ca'$ under the input $u$ is equal to the transition
probability of $\Sigma$ moving from $a$ to the equivalence class
$[a'] = \{x \in \Delta_N \colon Cx = Ca'\}$ when $u$ is applied.
\end{enumerate}
\end{theorem}

\begin{proof}
We first claim that for all $u \in \Delta_M$ and $a, a' \in
\Delta_N$ we have
\begin{equation} \label{eq4a}
\sum_{x \in [a']} x^{\top} P(u) a = (q')^{\top} \widetilde{P}(u) q ,
\end{equation}
where $q = Ca$ and $q' = C a'$. To see this, suppose that $q =
\delta_{\widetilde{N}}^j$, $q' = \delta_{\widetilde{N}}^i$, and
$\widetilde{C} \delta_{\widetilde{N}}^j = \delta_N^s$. Then
\begin{align} \label{eq4b}
(q')^{\top} \widetilde{P}(u) q &= (\widetilde{P}(u))_{ij} =
\sum_{l=1}^N \bigg(\sum_{k=1}^N (C)_{ik}
(P(u))_{kl}\bigg)(\widetilde{C})_{lj} \notag \\
&= \sum_{k=1}^N (C)_{ik} (P(u))_{ks} .
\end{align}
The last equality follows since $(\widetilde{C})_{lj} = 1$ exactly
when $l = s$. Noting the equivalence
\begin{align*}
(C)_{ik} = 1 &\Longleftrightarrow C \delta_N^k =
\delta_{\widetilde{N}}^i = q' = C a' \Longleftrightarrow
(\delta_N^k, a') \in \mathcal{R} \\
&\Longleftrightarrow \delta_N^k \in [a'] ,
\end{align*}
we get the above (\ref{eq4b}) equal to
\begin{equation} \label{eq5}
\sum_{ \left\{ k \colon \delta_N^k \in [a'] \right\} } (P(u))_{ks} =
\sum_{\delta_N^k \in [a']} (\delta_N^k)^{\top} P(u) \delta_N^s .
\end{equation}
Since $\widetilde{C} \leq C^{\top}$ and $(\widetilde{C})_{sj} = 1$,
we have $(C)_{js} = 1$. Thus, $C \delta_N^s =
\delta_{\widetilde{N}}^j = q = Ca$ and, hence, $(\delta_N^s, a) \in
\mathcal{R}$. By (\ref{eq4}), the right-hand side of (\ref{eq5}) is
then equal to $\sum_{x \in [a']} x^{\top} P(u) a$, and the claim is
proved.

We can now prove (a) and (b). Let $u \in \Delta_M$ and $1 \leq j
\leq \widetilde{N}$ be fixed. It follows from (\ref{eq4a}) that
\begin{align} \label{eq6}
\sum_{i=1}^{\widetilde{N}} (\widetilde{P}(u))_{ij} &=
\sum_{i=1}^{\widetilde{N}} (\delta_{\widetilde{N}}^i)^{\top}
\widetilde{P}(u) \delta_{\widetilde{N}}^j \notag \\
&= \sum_{i=1}^{\widetilde{N}} \sum_{\big\{ x \colon Cx =
\delta_{\widetilde{N}}^i \big\} } x^{\top} P(u) \delta_N^r ,
\end{align}
where $1 \leq r \leq N$ is such that $C \delta_N^r =
\delta_{\widetilde{N}}^j$ (such an $r$ exists since $C \in
\mathcal{L}^{\widetilde{N} \times N}$ is of full row rank). Since
$\Delta_N$ is the disjoint union of the sets $\big\{x \colon Cx =
\delta_{\widetilde{N}}^i \big\}$, $i = 1, \ldots, \widetilde{N}$,
the above (\ref{eq6}) is equal to $\sum_{k=1}^N (\delta_N^k)^{\top}
P(u) \delta_N^r = \sum_{k=1}^N (P(u))_{kr} = 1$, where the final
equality follows from the column-stochasticity of $P(u)$. This shows
that $\widetilde{P}(u)$ is column-stochastic, proving (a).

In order to prove part (b), we note that the right-hand side of
(\ref{eq4a}) is exactly the transition probability of
$\Sigma_{\mathcal{R}}$ from $q = Ca$ to $q' = Ca'$ under input $u$.
On the other hand, the left-hand side of (\ref{eq4a}) is the
transition probability with which $\Sigma$ moves from $a$ to
equivalence class $[a']$ when control action $u$ is applied. The
assertion of part (b) then follows from (\ref{eq4a}).
\end{proof}

Since, by the above theorem, $\Sigma_{\mathcal{R}}$ generates the
transitions of $\mathcal{T}(\Sigma) / \mathcal{R}$ (recall that the
assignment $[x] \mapsto Cx$ is a bijection between $\Delta_N /
\mathcal{R}$ and $\Delta_{\widetilde{N}}$), it can be interpreted as
a quotient of the PBN $\Sigma$.

\begin{remark} \label{rmk1}
Note that for a given $u \in \Delta_M$, the matrix
$\widetilde{P}(u)$ introduced in Theorem~\ref{thm1} is a constant
for all $\widetilde{C} \in \mathcal{L}^{N \times \widetilde{N}}$
such that $\widetilde{C} \leq C^\top$. Indeed, it follows from
(\ref{eq4b}) and (\ref{eq5}) that the $(i,j)$-entry of
$\widetilde{P}(u)$ is equal to the probability of $\Sigma$ moving
from the state $\widetilde{C} \delta_{\widetilde{N}}^j \in \Delta_N$
to the equivalence class $\big\{x \in \Delta_N \colon Cx =
\delta_{\widetilde{N}}^i \big\}$ when input $u$ is applied. It is
easy to see that for any $\widetilde{C} \in \mathcal{L}^{N \times
\widetilde{N}}$ satisfying $\widetilde{C} \leq C^\top$,
$\widetilde{C} \delta_{\widetilde{N}}^j$ belongs to the equivalence
class $\big\{x \in \Delta_N \colon Cx = \delta_{\widetilde{N}}^j
\big\}$. Since all states in $\big\{x \colon Cx =
\delta_{\widetilde{N}}^j \big\}$ have the same probability of
transitioning into $\big\{x \colon Cx = \delta_{\widetilde{N}}^i
\big\}$ given input $u$ (cf. (\ref{eq4})), the $(i,j)$-entry of
$\widetilde{P}(u)$ is constant for all logical matrices
$\widetilde{C}$ such that $\widetilde{C} \leq C^\top$, from which we
conclude that $\widetilde{P}(u)$ is a constant matrix whenever
$\widetilde{C} \leq C^\top$.
\end{remark}

\begin{example} \label{eg1}
As a simple illustration of Theorem~\ref{thm1}, consider a PBN as in
(\ref{eq2}), with $N = 8$, $M = 2$, and the transition probability
matrix given by
\begin{align*}
P &= \big[ \delta_8^2 \;\; 0.5 \delta_8^1 + 0.5 \delta_8^3 \;\; 0.5
\delta_8^1 + 0.5 \delta_8^3 \;\; \delta_8^5 \;\; \delta_8^6 \;\;
\delta_8^7 \;\; \delta_8^8 \;\; \delta_8^5 \, | \\
& \quad \;\;\; \delta_8^1 \;\; \delta_8^1 \;\; \delta_8^1 \;\; 0.3
\delta_8^4 + 0.7 \delta_8^8 \;\; \delta_8^6 \;\; \delta_8^7 \;\; 0.5
\delta_8^6 +
0.5 \delta_8^8 \;\; \delta_8^7 \big] \\
&= \big[ P(\delta_2^1) \;\; P(\delta_2^2) \big] .
\end{align*}
The state transition diagram of the PBN is shown in Fig.~\ref{fig1}.
\begin{figure}
\centerline{\includegraphics[width= 7.8 cm]{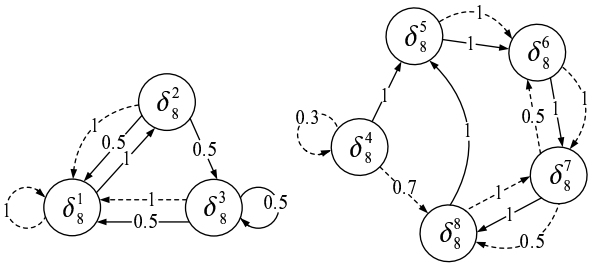}}
\caption{State transition diagram of the PBN in Example~\ref{eg1}. A
solid arrow represents the transition by the input $\delta_2^1$ and
a dashed arrow represents the transition by the input $\delta_2^2$.
The number associated with each arrow denotes the probability of the
state transition given the input.}  \label{fig1}
\end{figure}
Let $\mathcal{R}$ be the equivalence relation on $\Delta_8$ produced
by the partition $\{ \{\delta_8^1\}, \, \{\delta_8^2, \delta_8^3\},
\, \{\delta_8^4\}, \, \{\delta_8^5, \delta_8^6, \delta_8^7,
\delta_8^8\} \}$ (that is, the pair $(x,x') \in \mathcal{R}$ exactly
when $x$ and $x'$ belong to the same subset of the partition). It is
easily checked that (\ref{eq4}) is satisfied. The matrix
representing $\mathcal{R}$ is
\begin{equation*}
A_{\mathcal{R}} =
\begin{bmatrix}
1 & 0 & 0 & 0 \\
0 & J_2 & 0 & 0 \\
0 & 0 & 1 & 0 \\
0 & 0 & 0 & J_4
\end{bmatrix} ,
\end{equation*}
where $J_k$ denotes the all-one matrix of size $k \times k$.
Collapsing the identical rows of $A_{\mathcal{R}}$ yields a full row
rank matrix
\begin{equation*}
C = \big[ \delta_4^1 \;\; \delta_4^2 \;\; \delta_4^2 \;\; \delta_4^3
\;\; \delta_4^4 \;\; \delta_4^4 \;\; \delta_4^4 \;\; \delta_4^4
\big]
\end{equation*}
which fulfills (\ref{eq3}); and we take $\widetilde{C} = \big[
\delta_8^1 \;\; \delta_8^2 \;\; \delta_8^4 \;\; \delta_8^5 \big]$,
which satisfies $\widetilde{C} \leq C^{\top}$. A calculation then
yields
\begin{align*}
\widetilde{P}(\delta_2^1) &= C P(\delta_2^1) \widetilde{C} =
\big[\delta_4^2 \;\; 0.5 \delta_4^1 + 0.5 \delta_4^2 \;\; \delta_4^4
\;\; \delta_4^4 \big] , \\
\widetilde{P}(\delta_2^2) &= C P(\delta_2^2) \widetilde{C} =
\big[\delta_4^1 \;\; \delta_4^1 \;\; 0.3 \delta_4^3 + 0.7 \delta_4^4
\;\; \delta_4^4 \big] .
\end{align*}
The state transition diagram of $\Sigma_{\mathcal{R}}$ whose
transition probability matrix is given by $\widetilde{P} = \big[
\widetilde{P}(\delta_2^1) \;\; \widetilde{P}(\delta_2^2) \big]$ is
shown in Fig.~\ref{fig2}.
\begin{figure}
\centerline{\includegraphics[width= 7.5 cm]{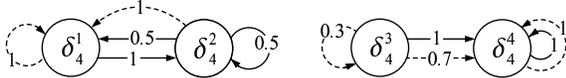}}
\caption{State transition diagram of $\Sigma_{\mathcal{R}}$ defined
in Example~\ref{eg1}.} \label{fig2}
\end{figure}
It is clear from the figure that $\Sigma_{\mathcal{R}}$ is indeed a
quotient of the original network which does not distinguish between
states related by $\mathcal{R}$.
\end{example}

Theorem~\ref{thm1} enables us to obtain a quotient Boolean system
once an equivalence relation satisfying (\ref{eq4}) is found. For
the remainder of this section, we will discuss the issue of
computing equivalence relations which allow the construction of
quotient Boolean systems. More precisely we consider the following
problem: given a PBN $\Sigma$ and an equivalence relation
$\mathcal{S}$ on $\Delta_N$, determine the maximal (with respect to
set inclusion) equivalence relation $\mathcal{R} \subseteq \Delta_N
\times \Delta_N$ such that $\mathcal{R} \subseteq \mathcal{S}$ and
condition (\ref{eq4}) holds. Here, the relation $\mathcal{S}$ may be
interpreted as a preliminary classification of the states of
$\Sigma$; and we focus on finding the maximal equivalence relation
since in many cases we want the size of a quotient system to be as
small as possible. The following theorem suggests a way of deriving
such an equivalence relation.

\begin{theorem} \label{thm2}
Let $\Sigma$ be a PBN described by (\ref{eq2}) and let $\mathcal{S}$
be an equivalence relation on $\Delta_N$. Define a sequence of
relations $\mathcal{R}_k$ by
\begin{equation*}
\mathcal{R}_1 = \mathcal{S} \;\; \text{and} \;\; \mathcal{R}_{k+1} =
\bigg( \bigcap_{u \in \Delta_M} \mathcal{S}_{u,k} \bigg) \cap
\mathcal{R}_k ,
\end{equation*}
where $\mathcal{S}_{u,k}$ is the relation on $\Delta_N$ defined by:
$(a, a') \in \mathcal{S}_{u,k}$ if and only if $\sum_{x \in [b]}
x^{\top} P(u) a = \sum_{x \in [b]} x^{\top} P(u) a'$ for all $[b]
\in \Delta_N / \mathcal{R}_k$, with the matrix $P(u)$ given by
(\ref{eq2a}). Then:
\begin{enumerate} [(a)]
\item The sequence of relations $\mathcal{R}_1, \mathcal{R}_2,
\ldots, \mathcal{R}_k , \ldots$ satisfies $\mathcal{R}_1 \supseteq
\mathcal{R}_2 \supseteq \cdots \supseteq \mathcal{R}_k \supseteq
\cdots$.

\item There is an integer $k^\ast$ such that $\mathcal{R}_{k^\ast +
1} = \mathcal{R}_{k^\ast}$.

\item $\mathcal{R}_{k^\ast}$ is nonempty and is the maximal
equivalence relation on $\Delta_N$ such that $\mathcal{R}_{k^\ast}
\subseteq \mathcal{S}$ and property (\ref{eq4}) holds.
\end{enumerate}
\end{theorem}

\begin{proof}
We first note that, since $\mathcal{R}_1 = \mathcal{S}$ is an
equivalence relation, a simple inductive argument shows that for
each $k \geq 1$, $\mathcal{R}_k$ is also an equivalence relation and
the quotient $\Delta_N / \mathcal{R}_k$ in the definition of
$\mathcal{S}_{u,k}$ makes sense.

Part (a) is trivial. Part (b) follows from (a) and the finiteness of
each $\mathcal{R}_k$. We proceed to the proof of (c). The relation
$\mathcal{R}_{k^\ast}$ is clearly nonempty (since it contains the
identity relation on $\Delta_N$) and is a subset of $\mathcal{S}$.
To show that (\ref{eq4}) holds true, suppose $(a, a') \in
\mathcal{R}_{k^\ast}$, $[b] \in \Delta_N / \mathcal{R}_{k^\ast}$ and
$u \in \Delta_M$. Since $\mathcal{R}_{k^\ast} = \mathcal{R}_{k^\ast
+ 1} \subseteq \mathcal{S}_{u,k^\ast}$, it follows, from the
definition of $\mathcal{S}_{u,k^\ast}$, that $\sum_{x \in [b]}
x^\top P(u) a = \sum_{x \in [b]} x^\top P(u) a'$, showing that
(\ref{eq4}) holds for $\mathcal{R}_{k^\ast}$.

To prove the maximality of $\mathcal{R}_{k^\ast}$, let $\mathcal{R}
\subseteq \Delta_N \times \Delta_N$ be another equivalence relation
which is contained in $\mathcal{S}$ and satisfies (\ref{eq4}). We
show by induction that $\mathcal{R} \subseteq \mathcal{R}_k$ for all
$k$. This, in particular, means that $\mathcal{R} \subseteq
\mathcal{R}_{k^\ast}$, thus proving the maximality of
$\mathcal{R}_{k^\ast}$. The case $k = 1$ is trivial, so we take $k >
1$ and assume that $\mathcal{R} \subseteq \mathcal{R}_{k-1}$. Let
$(a,a') \in \mathcal{R}$ and fix $u \in \Delta_M$. Then we have
$\sum_{x \in E} x^\top P(u) a = \sum_{x \in E} x^\top P(u) a'$ for
any equivalence class $E$ of $\mathcal{R}$. Since $\mathcal{R}
\subseteq \mathcal{R}_{k-1}$, each equivalence class in
$\mathcal{R}_{k-1}$ is a disjoint union of equivalence classes of
$\mathcal{R}$. It follows that $\sum_{x \in [b]} x^\top P(u) a =
\sum_{x \in [b]} x^\top P(u) a'$ for all $[b] \in \Delta_N /
\mathcal{R}_{k-1}$, and consequently $(a,a') \in \mathcal{S}_{u,
k-1}$ by the definition of $\mathcal{S}_{u, k-1}$. Since $u \in
\Delta_M$ is arbitrary, we have $(a, a') \in \bigcap_{u \in
\Delta_M} \mathcal{S}_{u, k-1}$, and noting that $(a,a') \in
\mathcal{R} \subseteq \mathcal{R}_{k-1}$ we conclude $(a,a') \in
\mathcal{R}_k$. This shows $\mathcal{R} \subseteq \mathcal{R}_k$,
and the theorem is proved.
\end{proof}

Recall that a relation $\mathcal{R} \subseteq \Delta_N \times
\Delta_N$ can be represented by a $(0,1)$-matrix of size $N \times
N$, whose $(i,j)$-entry is $1$ if and only if $(\delta_N^i,
\delta_N^j) \in \mathcal{R}$. For the sake of applications, it is
convenient to reformulate the above theorem in terms of
$(0,1)$-matrices.

\begin{corollary} \label{cor1}
Suppose that $\mathcal{S}$ is an equivalence relation on $\Delta_N$
represented by a matrix $A_{\mathcal{S}}$. For each $u \in \Delta_M$
let $P(u)$ be as in (\ref{eq2a}). Define a sequence of
$(0,1)$-matrices by
\begin{equation*}
A_1 = A_{\mathcal{S}} \;\; \text{and} \;\; A_{k+1} = A_k \wedge
B_{k,1} \wedge \cdots \wedge B_{k,M} ,
\end{equation*}
where $B_{k,l}$ ($l = 1,2, \ldots, M$) are $N \times N$
$(0,1)$-matrices whose $(i,j)$-entry is $1$ if and only if the $i$th
and $j$th columns of $A_k P(\delta_M^l)$ are identical. Then there
is an integer $k^\ast$ such that $A_{k^\ast + 1} = A_{k^\ast}$, and
$A_{k^\ast}$ is the matrix representing the maximal equivalence
relation on $\Delta_N$ that is contained in $\mathcal{S}$ and
satisfies property (\ref{eq4}).
\end{corollary}

\begin{proof}
We show that, for each $k \geq 1$, the matrix $A_k$ represents the
equivalence relation $\mathcal{R}_k$ defined in Theorem~\ref{thm2};
the result then follows by Theorem~\ref{thm2}. We proceed by
induction on $k$, with the case $k=1$ being trivial. Suppose that
$\mathcal{R}_{k-1}$ has the matrix representation $A_{k-1}$. For $u
\in \Delta_M$ and $1 \leq r, s \leq N$, the $(r,s)$-entry of the
matrix $A_{k-1} P(u)$ is
\begin{align*}
\sum_{r' = 1}^N (A_{k-1})_{r r'} (P(u))_{r' s} &= \sum_{r' = 1}^N
(A_{k-1})_{r r'} (\delta_N^{r'})^\top P(u) \delta_N^s \\
&= \sum_{ \{r' \colon (A_{k-1})_{r r'} = 1\}} (\delta_N^{r'})^\top
P(u) \delta_N^s ,
\end{align*}
and since $\mathcal{R}_{k-1}$ is represented by $A_{k-1}$, this
equals
\begin{equation*}
\sum_{ \{r' : (\delta_N^r, \delta_N^{r'}) \in \mathcal{R}_{k-1} \}}
\negthickspace (\delta_N^{r'})^\top P(u) \delta_N^s = \negthickspace
\sum_{ \{x : (\delta_N^r, x) \in \mathcal{R}_{k-1} \}}
\negthickspace x^\top P(u) \delta_N^s .
\end{equation*}
Consequently, the $i$th and $j$th columns of $A_{k-1} P(u)$ are the
same exactly when
\begin{equation*}
\sum_{ \{x \colon (\delta_N^r, x) \in \mathcal{R}_{k-1} \} } x^\top
P(u) \delta_N^i = \sum_{ \{x \colon (\delta_N^r, x) \in
\mathcal{R}_{k-1} \} } x^\top P(u) \delta_N^j
\end{equation*}
for all $1 \leq r \leq N$, and the latter is clearly equivalent to
saying that $\sum_{x \in [b]} x^\top P(u) \delta_N^i = \sum_{x \in
[b]} x^\top P(u) \delta_N^j$ for each $[b] \in \Delta_N /
\mathcal{R}_{k-1}$. Hence, if $\mathcal{S}_{u, k-1}$ is the relation
described in Theorem~\ref{thm2} and if $u = \delta_M^l$, then
\begin{equation*}
(\delta_N^i, \delta_N^j) \in \mathcal{S}_{u, k-1}
\Longleftrightarrow \text{the $(i,j)$-entry of $B_{k-1,l}$ is $1$} ,
\end{equation*}
and thus $B_{k-1, l}$ is the matrix representing $\mathcal{S}_{u,
k-1}$. Observe that the matrix representation of the intersection of
relations is equal to the meet of the matrices representing these
relations (see, e.g., \cite[Section 9.3]{rosenbook}). We conclude
that the relation $\mathcal{R}_k$ is represented by $A_k$, and this
completes the proof.
\end{proof}

\begin{example} \label{eg2}
Consider again the PBN in Example~\ref{eg1}. If we let $\mathcal{S}$
be the equivalence relation determined by the partition $\mathcal{P}
= \{ \{\delta_8^1\}, \, \{\delta_8^2, \delta_8^3, \delta_8^4\}, \,
\{\delta_8^5, \delta_8^6, \delta_8^7, \delta_8^8\} \}$, then
\begin{equation*}
A_1 =
\begin{bmatrix}
1 & 0 & 0 \\
0 & J_3 & 0 \\
0 & 0 & J_4
\end{bmatrix} ,
\end{equation*}
and a direct computation from Corollary~\ref{cor1} yields
\begin{equation*}
A_2 = A_3 =
\begin{bmatrix}
1 & 0 & 0 & 0 \\
0 & J_2 & 0 & 0 \\
0 & 0 & 1 & 0 \\
0 & 0 & 0 & J_4
\end{bmatrix} ,
\end{equation*}
which is precisely the matrix representing the relation given in
Example~\ref{eg1}. Hence the relation $\mathcal{R}$ presented in
Example~\ref{eg1} is the maximal equivalence relation contained in
$\mathcal{S}$ which satisfies condition (\ref{eq4}). We mention that
here it is easy to check directly that the obtained $\mathcal{R}$ is
indeed maximal. Specifically, note that any equivalence relation
contained in $\mathcal{S}$ corresponds to a refinement of the
partition $\mathcal{P} = \{ \{\delta_8^1\}, \, \{\delta_8^2,
\delta_8^3, \delta_8^4\}, \, \{\delta_8^5, \delta_8^6, \delta_8^7,
\delta_8^8\} \}$. Since, for $u \in \Delta_2$, $(\delta_8^1)^\top
P(u) \delta_8^2 = (\delta_8^1)^\top P(u) \delta_8^3 \neq 0$ while
$(\delta_8^1)^\top P(u) \delta_8^4 = 0$, condition (\ref{eq4}) does
not hold for any equivalence relation corresponding to a refinement
of $\mathcal{P}$ in which $\delta_8^2$ and $\delta_8^4$, or
$\delta_8^3$ and $\delta_8^4$, belong to the same block. On the
other hand, we observed in Example~\ref{eg1} that the relation
$\mathcal{R}$ produced by the partition $\{ \{\delta_8^1\}, \,
\{\delta_8^2, \delta_8^3\}, \, \{\delta_8^4\}, \, \{\delta_8^5,
\delta_8^6, \delta_8^7, \delta_8^8\} \}$ fulfills (\ref{eq4}); thus
it is the maximal equivalence relation which is contained in
$\mathcal{S}$ and satisfies (\ref{eq4}).
\end{example}

To conclude, we would like to point out that the proposed method for
generating a quotient of a PBN is a natural extension of the
approach presented in \cite{r.li2020} for constructing a quotient of
a deterministic BN. Recall that a deterministic BN described by
\begin{equation*}
\Sigma ' \colon \;\, x(t+1) = F \ltimes u(t) \ltimes x(t) , \; x \in
\Delta_N , \; u \in \Delta_M , \; F \in \mathcal{L}^{N \times NM}
\end{equation*}
can be seen as a special case of (\ref{eq2}), with $\theta(t)$
having a constant value with probability one for all $t \geq 0$. So
the results of this section apply at once. For $u \in \Delta_M$, let
$\widetilde{F}(u)$ be defined as $\widetilde{P}(u)$ is in
Theorem~\ref{thm1} with $P(u)$ in place of $F(u) \coloneqq F \ltimes
u$. We note that $\widetilde{F}(u)$ has all nonnegative integer
entries, and since it is column-stochastic by Theorem~\ref{thm1}(a),
every column contains exactly one nonzero entry and the nonzero
entry equals $1$, i.e., $\widetilde{F}(u)$ is a logical matrix.
Also, recall that the $(i,j)$-entry of $\widetilde{P}(u)$ defined in
Theorem~\ref{thm1} is equal to the probability with which the
original network reaches the equivalence class $\big\{x \in \Delta_N
\colon Cx = \delta_{\widetilde{N}}^i \big\}$ from an arbitrary but
fixed state in $\big\{x \colon Cx = \delta_{\widetilde{N}}^j \big\}$
when $u$ is applied (cf. Remark~\ref{rmk1}). Translated to the
deterministic setting, this means that $(\widetilde{F}(u))_{ij} = 1$
if and only if there is a one-step transition of $\Sigma'$ from a
state in $\big\{x \colon Cx = \delta_{\widetilde{N}}^j \big\}$ to a
state in $\big\{x \colon Cx = \delta_{\widetilde{N}}^i \big\}$ under
input $u$. The quotient system
\begin{equation*}
x_{\mathcal{R}}(t+1) = \widetilde{F} \ltimes u(t) \ltimes
x_{\mathcal{R}}(t)
\end{equation*}
given by Theorem~\ref{thm1}, where $\widetilde{F} =
\big[\widetilde{F}(\delta_M^1) \; \cdots \;
\widetilde{F}(\delta_M^M) \big]$, then coincides precisely with the
one presented in \cite[Theorem 1]{r.li2020}, in which a state
$\delta_{\widetilde{N}}^j$ can make a transition to another state
$\delta_{\widetilde{N}}^i$ by applying an input exactly when that
input drives $\Sigma'$ from some state in $\big\{x \colon Cx =
\delta_{\widetilde{N}}^j \big\}$ to some state in $\big\{x \colon Cx
= \delta_{\widetilde{N}}^i \big\}$.

\section{Control Design Via Quotients}

\label{sec4}

This section illustrates the application of quotient systems for
control design. We consider two typical control problems in PBNs and
show how the problems can be solved through the use of a quotient
Boolean system.

\subsection{Stabilization}

\label{sec4.1}

Consider a PBN $\Sigma$ as in (\ref{eq2}) and let $P(u)$ be as in
(\ref{eq2a}), which gives the (one-step) transition probabilities of
$\Sigma$ under input $u \in \Delta_M$. A (time-invariant) feedback
controller is given by a map $\mathcal{U} \colon \Delta_N
\rightarrow \Delta_M$ so that if the present state is $x \in
\Delta_N$, then the controller selects the control input
$\mathcal{U}(x) \in \Delta_M$, resulting in the matrix
$P(\mathcal{U}(x))$ that determines the one-step transition
probabilities. Observe that when the present state is, say,
$\delta_N^i$, only the transition probabilities of leaving
$\delta_N^i$ are relevant and are given by the $i$th column of the
matrix $P(\mathcal{U}(\delta_N^i))$. We use $P_{\mathcal{U}}$ to
denote the matrix obtained by stacking such columns, i.e., the $i$th
column of $P_{\mathcal{U}}$ is the $i$th column of
$P(\mathcal{U}(\delta_N^i))$. It is easy to see that the evolution
of $\Sigma$ under the control of the state feedback controller
$\mathcal{U} \colon \Delta_N \rightarrow \Delta_M$ is governed by
the matrix $P_{\mathcal{U}}$, i.e., the transition probability from
$a \in \Delta_N$ to $b \in \Delta_N$ after $k$ steps is given by
$b^\top P_{\mathcal{U}}^k a$. Let $\mathcal{M} \subseteq \Delta_N$
be a target set of states. The Boolean system $\Sigma$ is stabilized
to $\mathcal{M}$ with probability one by $\mathcal{U} \colon
\Delta_N \rightarrow \Delta_M$, if for every initial state $x_0 \in
\Delta_N$, there exists an integer $\tau$ such that $k \geq \tau$
implies $\sum_{x \in \mathcal{M}} x^\top P_{\mathcal{U}}^k x_0 = 1$
(see, e.g., \cite{r.li2014b,zhou2020}). The following result shows
that we can easily derive a stabilizing controller for $\Sigma$ on
the basis of a stabilizing controller for its quotient system.

\begin{proposition} \label{prop1}
Consider a PBN $\Sigma$ as given in (\ref{eq2}). Let $\mathcal{M}
\subseteq \Delta_N$ and let $\mathcal{S}$ be the equivalence
relation on $\Delta_N$ determined by the partition $\left\{
\mathcal{M}, \Delta_N - \mathcal{M} \right\}$. Suppose that
$\mathcal{R}$ is an equivalence relation on $\Delta_N$ induced by a
full row rank matrix $C \in \mathcal{L}^{\widetilde{N} \times N}$,
$\mathcal{R} \subseteq \mathcal{S}$, and (\ref{eq4}) holds. Suppose
$\Sigma_{\mathcal{R}}$ is defined as in Theorem~\ref{thm1} and let
$\mathcal{M}_{\mathcal{R}} = \{ Cx \colon x \in \mathcal{M} \}$.
Then:
\begin{enumerate} [(a)]
\item There exists a control law $\mathcal{U} \colon \Delta_N
\rightarrow \Delta_M$ that stabilizes $\Sigma$ to $\mathcal{M}$ with
probability one if and only if there exists a control law
$\mathcal{U}_{\mathcal{R}} \colon \Delta_{\widetilde{N}} \rightarrow
\Delta_M$ that stabilizes $\Sigma_{\mathcal{R}}$ to
$\mathcal{M}_{\mathcal{R}}$ with probability one.

\item If the controller $x_\mathcal{R} \mapsto
\mathcal{U}_{\mathcal{R}}(x_{\mathcal{R}})$ stabilizes
$\Sigma_{\mathcal{R}}$ to $\mathcal{M}_{\mathcal{R}}$ with
probability one, then the controller given by $x \mapsto
\mathcal{U}(x) = \mathcal{U}_{\mathcal{R}} (Cx)$ stabilizes $\Sigma$
to $\mathcal{M}$ with probability one.
\end{enumerate}
\end{proposition}

For the proof of Proposition~\ref{prop1} we need the following lemma
adapted from \cite{r.li2016}. To make the paper self-contained, the
proof of this lemma is given in the Appendix.

\begin{lemma} \label{lem1a}
Consider a PBN as in (\ref{eq2}). Let $\mathcal{M} \subseteq
\Delta_N$, and let $\mathcal{M}^\ast$ be the last term of the
sequence
\begin{align*}
\mathcal{M}_0 &= \mathcal{M} , \\
\mathcal{M}_i &= \mathcal{M}_{i-1} \cap \mathcal{A}
(\mathcal{M}_{i-1}) , \quad i = 1, \ldots , \iota ,
\end{align*}
where $\mathcal{A}(\mathcal{M}_{i-1}) = \{a \in \Delta_N \colon
\sum_{x \in \mathcal{M}_{i-1}} x^\top P(u)a = 1 \: \text{for some}$
$u \in \Delta_M\}$, and the value of $\iota$ is determined by the
condition $\mathcal{M}_{\iota + 1} = \mathcal{M}_\iota$. Define the
sequence $\mathcal{Z}_j$ according to
\begin{align*}
\mathcal{Z}_0 &= \mathcal{M}^\ast , \\
\mathcal{Z}_j &= \Big\{a \in \Delta_N : \negthickspace \sum_{x \in
\mathcal{Z}_{j-1}} \negthickspace x^\top P(u)a = 1 \, \text{for} \,
\text{some} \, u \in \Delta_M \Big\} , \, j \geq 1 .
\end{align*}
Then $\mathcal{Z}_j \supseteq \mathcal{Z}_{j-1}$, and the PBN can be
stabilized to $\mathcal{M}$ with probability one by a feedback
$\mathcal{U} \colon \Delta_N \rightarrow \Delta_M$ if, and only if,
$\mathcal{Z}_\lambda = \Delta_N$ for some $\lambda \geq 1$.
\end{lemma}

\begin{proof} [Proof of Proposition~\ref{prop1}]
(a) Let $\mathcal{M}_i$ and $\mathcal{Z}_j$ be as in
Lemma~\ref{lem1a}. Let $\mathcal{M}_{\mathcal{R}}^\ast$ be the last
term of the sequence
\begin{align*}
\widetilde{\mathcal{M}}_0 &= \mathcal{M}_{\mathcal{R}} , \\
\widetilde{\mathcal{M}}_i &= \widetilde{\mathcal{M}}_{i-1} \cap
\mathcal{A}' (\widetilde{\mathcal{M}}_{i-1}) , \quad i = 1, \ldots ,
\iota ' ,
\end{align*}
where $\mathcal{A}' (\widetilde{\mathcal{M}}_{i-1}) = \{q \in
\Delta_{\widetilde{N}} \colon \sum_{z \in
\widetilde{\mathcal{M}}_{i-1}} z^\top \widetilde{P}(u)q = 1 \:
\text{for some}$ $u \in \Delta_M\}$, and the value of $\iota'$ is
determined by the condition $\widetilde{\mathcal{M}}_{\iota' + 1} =
\widetilde{\mathcal{M}}_{\iota'}$. Define the sequence
$\widetilde{\mathcal{Z}}_j$ according to
\begin{align*}
\widetilde{\mathcal{Z}}_0 &= \mathcal{M}_{\mathcal{R}}^\ast , \\
\widetilde{\mathcal{Z}}_j &= \bigg\{q \in \Delta_{\widetilde{N}} :
\negthickspace \sum_{z \in \widetilde{\mathcal{Z}}_{j-1}}
\negthickspace z^\top \widetilde{P}(u)q = 1 \, \text{for} \,
\text{some} \, u \in \Delta_M \bigg\} , \, j \geq 1 .
\end{align*}
We show that for $j \geq 0$,
\begin{equation} \label{eq20}
x \in \mathcal{Z}_j \Longleftrightarrow C x \in
\widetilde{\mathcal{Z}}_j .
\end{equation}
First, we claim that
\begin{equation} \label{eq21}
x \in \mathcal{M}_i \Longleftrightarrow C x \in
\widetilde{\mathcal{M}}_i .
\end{equation}
Indeed, if $Cx \in \widetilde{\mathcal{M}}_0$, then there exists $x'
\in \mathcal{M}$ such that $Cx = Cx'$, and hence $(x,x') \in
\mathcal{R} \subseteq \mathcal{S}$, forcing $x \in \mathcal{M}$
since $\mathcal{S}$ is the equivalence relation yielded by the
partition $\left\{ \mathcal{M}, \Delta_N - \mathcal{M} \right\}$.
This shows that $Cx \in \widetilde{\mathcal{M}}_0 \Rightarrow x \in
\mathcal{M}_0$. The converse implication is trivial. Assume by
induction that $x \in \mathcal{M}_{i-1} \Longleftrightarrow C x \in
\widetilde{\mathcal{M}}_{i-1}$. Denoting $I(z) = \{x \in \Delta_N
\colon Cx = z \}$ for $z \in \Delta_{\widetilde{N}}$, which is
nonempty since $C$ is supposed to have full row rank, then
$\mathcal{M}_{i-1}$ can be partitioned as the disjoint union
$\mathcal{M}_{i-1} = \bigcup_{z \in \widetilde{\mathcal{M}}_{i-1}}
I(z)$. Indeed, the sets $I(z)$, $z \in
\widetilde{\mathcal{M}}_{i-1}$, are clearly mutually disjoint, and
for any $x \in \Delta_N$, $x \in \mathcal{M}_{i-1}$ if and only if
$Cx \in \widetilde{\mathcal{M}}_{i-1}$, if and only if $x \in I(z)$
for some $z \in \widetilde{\mathcal{M}}_{i-1}$. Suppose $x \in
\Delta_N$, $u \in \Delta_M$, and let $q = Cx$. Then
\begin{equation*}
\sum_{b \in \mathcal{M}_{i-1}}b^\top P(u)x = \negthickspace \sum_{z
\in \widetilde{\mathcal{M}}_{i-1}} \sum_{b \in I(z)} b^\top P(u)x =
\negthickspace \sum_{z \in \widetilde{\mathcal{M}}_{i-1}}
\negthickspace z^\top \widetilde{P}(u)q ,
\end{equation*}
where the second equality follows from (\ref{eq4a}) in the proof of
Theorem~\ref{thm1}. This immediately implies that $x \in
\mathcal{A}(\mathcal{M}_{i-1})$ if and only if $Cx \in
\mathcal{A}'(\widetilde{\mathcal{M}}_{i-1})$, and hence $x \in
\mathcal{M}_i$ if and only if $Cx \in \widetilde{\mathcal{M}}_i$.

The proof of (\ref{eq20}) is easily obtained by induction on $j$. It
follows from (\ref{eq21}) that $x \in \mathcal{Z}_0$ if and only if
$Cx \in \widetilde{\mathcal{Z}}_0$, establishing the base step. The
induction step is similar to that done in the proof of (\ref{eq21}).

Since $C$ is of full row rank, we conclude from (\ref{eq20}) that
$\mathcal{Z}_j = \Delta_N$ if and only if $\widetilde{\mathcal{Z}}_j
= \Delta_{\widetilde{N}}$, and the proof of (a) follows by
Lemma~\ref{lem1a}.

(b) Define the matrix $\widetilde{P}_{\mathcal{U}_{\mathcal{R}}}$
for $\Sigma_{\mathcal{R}}$ in the same way as $P_{\mathcal{U}}$ is
defined for $\Sigma$. We first prove that, for any $a \in \Delta_N$,
$z \in \Delta_{\widetilde{N}}$, and integer $k \geq 1$, we have
\begin{equation} \label{eq7}
\sum_{x \in I(z)} x^\top P_{\mathcal{U}}^k a = z^\top
\widetilde{P}_{\mathcal{U}_{\mathcal{R}}}^k q ,
\end{equation}
where $I(z) = \{x \in \Delta_N \colon Cx = z \}$ and $q = Ca$. The
proof is by induction on $k$. Since $P_{\mathcal{U}} a =
P(\mathcal{U}(a))a$ by the construction of $P_{\mathcal{U}}$, it
follows from (\ref{eq4a}) in the proof of Theorem~\ref{thm1} that
\begin{equation*}
\sum_{x \in I(z)} x^\top P_{\mathcal{U}} a = \sum_{x \in I(z)}
x^\top P(\mathcal{U}(a)) a = z^\top \widetilde{P}(\mathcal{U}(a)) q
,
\end{equation*}
and since $\mathcal{U}(a) = \mathcal{U}_{\mathcal{R}} (Ca) =
\mathcal{U}_{\mathcal{R}}(q)$, the above is equal to $z^\top
\widetilde{P}(\mathcal{U}_{\mathcal{R}}(q))q = z^\top
\widetilde{P}_{\mathcal{U}_{\mathcal{R}}}q$. This gives (\ref{eq7})
for $k = 1$. Assume as induction hypothesis that the statement holds
for $k-1$. Decomposing the $N \times N$ identity matrix as $\sum_{b
\in \Delta_N} b b^\top$, we have
\begin{align} \label{eq8}
&\sum_{x \in I(z)} x^\top P_{\mathcal{U}}^k a = \sum_{x \in I(z)}
x^\top P_{\mathcal{U}} \bigg( \sum_{b \in \Delta_N} b b^\top \bigg)
P_{\mathcal{U}}^{k-1} a \notag \\
&\qquad = \sum_{x \in I(z)} \sum_{b \in \Delta_N} x^\top
P_{\mathcal{U}} b
b^\top P_{\mathcal{U}}^{k-1} a \notag \\
&\qquad = \sum_{i=1}^{\widetilde{N}} \sum_{b \in
I(\delta_{\widetilde{N}}^i)} \bigg(\sum_{x \in I(z)} x^\top
P_{\mathcal{U}} b \bigg) b^\top P_{\mathcal{U}}^{k-1} a .
\end{align}
The last equality holds true since $\Delta_N$ is the disjoint union
of the sets $I(\delta_{\widetilde{N}}^i) = \{x \colon Cx =
\delta_{\widetilde{N}}^i \}$, $i = 1, \ldots, \widetilde{N}$. It
follows from the case $k=1$ that $\sum_{x \in I(z)} x^\top
P_{\mathcal{U}} b = z^\top \widetilde{P}_{\mathcal{U}_{\mathcal{R}}}
\delta_{\widetilde{N}}^i$ for all $b \in
I(\delta_{\widetilde{N}}^i)$, and the right-hand side of (\ref{eq8})
is equal to the following expression:
\begin{equation} \label{eq11a}
\sum_{i=1}^{\widetilde{N}} z^\top
\widetilde{P}_{\mathcal{U}_{\mathcal{R}}} \delta_{\widetilde{N}}^i
\bigg( \sum_{b \in I(\delta_{\widetilde{N}}^i)} b^\top
P_{\mathcal{U}}^{k-1} a \bigg) .
\end{equation}
According to the induction hypothesis, we have for each $1 \leq i
\leq \widetilde{N}$,
\begin{equation*}
\sum_{b \in I(\delta_{\widetilde{N}}^i)}b^\top
P_{\mathcal{U}}^{k-1}a = (\delta_{\widetilde{N}}^i)^\top
\widetilde{P}_{\mathcal{U}_{\mathcal{R}}}^{k-1}q ,
\end{equation*}
and substituting this into (\ref{eq11a}) we get
\begin{align*}
&\sum_{x \in I(z)} x^\top P_{\mathcal{U}}^k a =
\sum_{i=1}^{\widetilde{N}} z^\top
\widetilde{P}_{\mathcal{U}_{\mathcal{R}}} \delta_{\widetilde{N}}^i
\Big[(\delta_{\widetilde{N}}^i)^\top
\widetilde{P}_{\mathcal{U}_{\mathcal{R}}}^{k-1} q \Big] \\
&\qquad = z^\top \widetilde{P}_{\mathcal{U}_{\mathcal{R}}} \bigg[
\sum_{i=1}^{\widetilde{N}} \delta_{\widetilde{N}}^i
(\delta_{\widetilde{N}}^i)^\top \bigg]
\widetilde{P}_{\mathcal{U}_{\mathcal{R}}}^{k-1} q = z^\top
\widetilde{P}_{\mathcal{U}_{\mathcal{R}}}^k q ,
\end{align*}
which is (\ref{eq7}).

From the proof of (a), we know that $x \in \mathcal{M}$ if and only
if $Cx \in \mathcal{M}_{\mathcal{R}}$, and consequently, we can
write $\mathcal{M}$ as the disjoint union $\mathcal{M} = \bigcup_{z
\in \mathcal{M}_{\mathcal{R}}} I(z)$. The proof of part (b) is now
obvious. Suppose $x_0 \in \Delta_N$. Let $x_{\mathcal{R}}^0 = C
x_0$. Then for each integer $k \geq 1$ we have
\begin{equation*}
\sum_{x \in \mathcal{M}} x^\top P_{\mathcal{U}}^k x_0 = \sum_{z \in
\mathcal{M}_{\mathcal{R}}} \sum_{x \in I(z)} x^\top
P_{\mathcal{U}}^k x_0 = \sum_{z \in \mathcal{M}_{\mathcal{R}}}
z^\top \widetilde{P}_{\mathcal{U}_{\mathcal{R}}}^k x_{\mathcal{R}}^0
,
\end{equation*}
from which part (b) follows immediately.
\end{proof}

\subsection{Optimal Control}

\label{sec4.2}

Let us consider the following optimal control problem, introduced in
\cite{datta2003}.

\begin{problem} \label{plm1}
Consider a PBN as in (\ref{eq2}). Given an initial state $x_0 \in
\Delta_N$ and a finite time horizon $T \in \mathbb{Z}^+$, find a
control policy, $u(t) = \mathcal{U}^\ast (t, x(t))$ for $0 \leq t
\leq T-1$, that minimizes the cost functional
\begin{equation*}
J = \mathbb{E} \bigg[ g(x(T)) + \sum_{t=0}^{T-1} l(u(t), x(t))
\bigg] ,
\end{equation*}
where $l(u,x)$ and $g(x)$ are real-valued functions defined on
$\Delta_M \times \Delta_N$ and $\Delta_N$, respectively.
\end{problem}

We show that the solution to Problem~\ref{plm1} can be found by
considering the problem for a suitably chosen quotient system. To
this end, let $\mathcal{S}$ be the equivalence relation on
$\Delta_N$ given by
\begin{align} \label{eq8a}
(x, x') \in \mathcal{S} &\Longleftrightarrow \text{$g(x) = g(x')$
and} \notag \\
& \qquad \;\, \text{$l(u, x) = l(u, x')$ for all $u \in \Delta_M$}.
\end{align}
We note that, if $C \in \mathcal{L}^{\widetilde{N} \times N}$ has
full row rank, and if the equivalence relation $\mathcal{R}$ induced
by $C$ satisfies $\mathcal{R} \subseteq \mathcal{S}$, then every $z
\in \Delta_{\widetilde{N}}$ can be written as $z = Cx$ for some $x
\in \Delta_N$ and the function $g$ is constant on the set $I(z) =
\{x \in \Delta_N \colon Cx = z\}$. Hence, the map $g_{\mathcal{R}}
\colon \Delta_{\widetilde{N}} \rightarrow \mathbb{R}$, given by
\begin{equation} \label{eq9}
g_{\mathcal{R}}(z) = g_{\mathcal{R}} (Cx) = g(x) ,
\end{equation}
is well defined. For the same reason, the map $l_{\mathcal{R}}
\colon \Delta_M \times \Delta_{\widetilde{N}} \rightarrow
\mathbb{R}$ defined by
\begin{equation} \label{eq10}
l_{\mathcal{R}}(u, z) = l_{\mathcal{R}}(u, Cx) = l(u, x)
\end{equation}
is also well defined. We can state the following proposition.

\begin{proposition} \label{prop2}
Let $\Sigma$ be a PBN described by (\ref{eq2}) and consider
Problem~\ref{plm1} with given $x_0$ and $T$. Suppose that
$\mathcal{S}$ is the equivalence relation given by (\ref{eq8a}),
$\mathcal{R} \subseteq \Delta_N \times \Delta_N$ is an equivalence
relation induced by a full row rank matrix $C \in
\mathcal{L}^{\widetilde{N} \times N}$, $\mathcal{R} \subseteq
\mathcal{S}$, and (\ref{eq4}) holds. Let $\Sigma_{\mathcal{R}}$ be
the probabilistic Boolean system constructed in Theorem~\ref{thm1},
and define $J_{\mathcal{R}} = \mathbb{E} \, \big[ g_\mathcal{R}
(x_{\mathcal{R}}(T)) + \sum_{t=0}^{T-1} l_{\mathcal{R}}(u(t),
x_{\mathcal{R}}(t)) \big]$, where $g_{\mathcal{R}}$ and
$l_{\mathcal{R}}$ are given by (\ref{eq9}) and (\ref{eq10}). Suppose
that $(t, x_{\mathcal{R}}) \mapsto \mathcal{U}_{\mathcal{R}}^\ast
(t, x_{\mathcal{R}})$ is an optimal control policy solving
Problem~\ref{plm1} with $\Sigma$, $x_0$, and $J$ replaced by
$\Sigma_{\mathcal{R}}$, $x_{\mathcal{R}}^0 = C x_0$, and
$J_{\mathcal{R}}$, respectively. Then the control policy given by
$(t, x) \mapsto \mathcal{U}^\ast (t, x) =
\mathcal{U}_{\mathcal{R}}^\ast (t, Cx)$ is an optimal control policy
for $\Sigma$. Moreover, let $J^\ast$ be the optimal value of $J$
given the initial state $x_0$ and let $J_{\mathcal{R}}^\ast$ be the
optimal value of $J_{\mathcal{R}}$ associated with
$x_{\mathcal{R}}^0 = C x_0$. Then $J^\ast = J_{\mathcal{R}}^\ast$.
\end{proposition}

The proof of the proposition follows from the following two lemmas.

\begin{lemma} \label{lem1}
Consider Problem~\ref{plm1} with given $x_0$ and $T$. Let
$\mathcal{S}$, $\mathcal{R}$, and $C$ be as in
Proposition~\ref{prop2}. Then there exists an optimal control policy
$(t, x) \mapsto \overline{\mathcal{U}} (t,x)$ with the property that
$\overline{\mathcal{U}}(t, x) = \overline{\mathcal{U}}(t, x')$ for
all $0 \leq t \leq T-1$ and all $x, x' \in \Delta_N$ such that $Cx =
Cx'$.
\end{lemma}

\begin{proof}
Consider the following dynamic programming algorithm (adapted from
\cite[Proposition 1.3.1]{bertsekasbook}; see also \cite{datta2003}):
\begin{align*}
&H(T, x) = g(x) , \quad x \in \Delta_N , \\
&H(t, x) = \min_{u \in \Delta_M} \bigg\{ l(u,x) + \sum_{\xi \in
\Delta_N} H(t+1, \xi) \xi^{\top} P(u) x \bigg\} , \\
&\qquad \qquad \qquad x \in \Delta_N , \quad t = T-1, \ldots, 1, 0,
\end{align*}
where $P(u)$ is as in (\ref{eq2a}). If we let
\begin{equation*}
G(t,x,u) = l(u,x) + \sum_{\xi \in \Delta_N} H(t+1, \xi) \xi^{\top}
P(u) x ,
\end{equation*}
and define
\begin{equation*}
\overline{\mathcal{U}}(t, x) \in \arg \min_{u \in \Delta_M} G(t, x,
u), \;\; 0 \leq t \leq T-1, \;\; x \in \Delta_N ,
\end{equation*}
then the control law given by $(t, x) \mapsto
\overline{\mathcal{U}}(t,x)$ is optimal
\cite{bertsekasbook,datta2003}. We will show that for $0 \leq t \leq
T-1$,
\begin{multline} \label{eq10a}
G(t,x,u) = G(t, x', u) , \quad \forall u \in \Delta_M , \;\; \forall
x, x' \in \Delta_N \\
\text{with} \; C x = C x' .
\end{multline}
Then we can find $\overline{\mathcal{U}}(t,x) \in \arg \min_u
G(t,x,u)$ with the desired property. This will prove the lemma.

Fix $u \in \Delta_M$ and let $x, x' \in \Delta_N$ be such that $Cx =
Cx'$. Since $(x,x') \in \mathcal{R} \subseteq \mathcal{S}$, it
follows from (\ref{eq8a}) that
\begin{equation} \label{eq11}
l(u,x) = l(u, x') .
\end{equation}
For each $z \in \Delta_{\widetilde{N}}$, since $H(T, \cdot) =
g(\cdot)$ is constant on the set $I(z) = \{\xi \in \Delta_N \colon C
\xi = z\}$ (cf. the statement following (\ref{eq8a})) and since
\begin{equation*}
\sum_{\xi \in I(z)} \xi^\top P(u) x = \sum_{\xi \in I(z)} \xi^\top
P(u) x'
\end{equation*}
by (\ref{eq4}), we have
\begin{equation*}
\sum_{\xi \in I(z)} H(T, \xi) \xi^\top P(u) x = \sum_{\xi \in I(z)}
H(T, \xi) \xi^\top P(u) x' .
\end{equation*}
Hence,
\begin{equation*}
\sum_{\xi \in \Delta_N} H(T, \xi) \xi^\top P(u) x = \sum_{\xi \in
\Delta_N} H(T, \xi) \xi^\top P(u) x' ,
\end{equation*}
since $\Delta_N$ is the disjoint union of $I(z)$, $z \in
\Delta_{\widetilde{N}}$. This together with (\ref{eq11}) gives
$G(T-1, x, u) = G(T-1, x', u)$. Thus (\ref{eq10a}) is true if $t =
T-1$.

Note that if $t \leq T-1$ and if (\ref{eq10a}) is true for $t$, then
for any $\xi, \xi' \in \Delta_N$ with $C \xi = C \xi'$, we have
\begin{equation*}
H(t, \xi) = \min_{u \in \Delta_M} G(t, \xi, u) = \min_{u \in
\Delta_M} G(t, \xi ' , u) = H(t, \xi') .
\end{equation*}
Thus with this $t$ fixed, the function $H(t, \cdot)$ is constant on
each of the sets $I(z) = \{\xi \colon C \xi = z\}$. Then by an
argument similar to that in the previous paragraph, we can show that
(\ref{eq10a}) is true for $t-1$ also, and so working by downward
induction on $t$, we conclude that (\ref{eq10a}) holds true for all
$0 \leq t \leq T-1$, as required. The proof is complete.%
\end{proof}

\begin{lemma} \label{lem2}
Let the notation be as in the statement of Proposition~\ref{prop2}.
If the initial states of $\Sigma$ and $\Sigma_{\mathcal{R}}$ satisfy
$C x_0 = x_{\mathcal{R}}^0$, and if the two control policies $(t,x)
\mapsto \mathcal{U}(t, x)$ and $(t, x_{\mathcal{R}}) \mapsto
\widetilde{\mathcal{U}}(t, x_{\mathcal{R}})$ satisfy $\mathcal{U}(t,
x) = \widetilde{\mathcal{U}}(t, Cx)$ for all $0 \leq t \leq T-1$ and
$x \in \Delta_N$, then the cost functionals $J$ and
$J_{\mathcal{R}}$ have the same value.
\end{lemma}

\begin{proof}
For each $0 \leq t \leq T-1$, let $P_t$ be the matrix whose $i$th
column is the $i$th column of the matrix $P(\mathcal{U}(t,
\delta_N^i))$, and let $\widetilde{P}_t$ be the matrix in which the
$j$th column is the $j$th column of
$\widetilde{P}(\widetilde{\mathcal{U}}(t,
\delta_{\widetilde{N}}^j))$. With a similar argument to that in
proving (\ref{eq7}), it is easy to see that for any $a \in
\Delta_N$, $z \in \Delta_{\widetilde{N}}$, and $1 \leq t \leq T$, we
have $\sum_{x \in I(z)} x^\top P_{t-1} P_{t-2} \cdots P_0 a = z^\top
\widetilde{P}_{t-1} \widetilde{P}_{t-2} \cdots \widetilde{P}_0 q$,
where $I(z) = \{x \in \Delta_N \colon Cx = z\}$ and $q = Ca$. Fix $1
\leq t \leq T-1$, and fix $s \in \{l(u,x) \colon (u,x) \in \Delta_M
\times \Delta_N \}$. Define
\begin{align*}
\mathcal{M}(t,s) &= \{x \in \Delta_N \colon l(\mathcal{U}(t, x), x)
= s \} , \\
\widetilde{\mathcal{M}}(t,s) &= \{x_{\mathcal{R}} \in
\Delta_{\widetilde{N}} \colon l_{\mathcal{R}}
(\widetilde{\mathcal{U}}(t, x_{\mathcal{R}}), x_{\mathcal{R}}) = s
\} .
\end{align*}
Since $l(\mathcal{U}(t,x),x) = l(\widetilde{\mathcal{U}}(t, Cx), x)
= l_{\mathcal{R}}(\widetilde{\mathcal{U}}(t, Cx), Cx)$, it follows
that $x \in \mathcal{M}(t,s)$ if and only if $Cx \in
\widetilde{\mathcal{M}}(t,s)$, and hence $\mathcal{M}(t,s)$ can be
written as the disjoint union $\mathcal{M}(t,s) = \bigcup_{z \in
\widetilde{\mathcal{M}}(t,s)} I(z)$. Consequently,
\begin{align*}
&\mathbb{P} \big\{l(\mathcal{U}(t, x(t)), x(t)) = s \big\} =
\mathbb{P} \big\{x(t) \in \mathcal{M}(t,s) \big\} \\
&= \negthickspace \sum_{z \in \widetilde{\mathcal{M}}(t,s)} \sum_{x
\in I(z)} \negthickspace x^\top P_{t-1} \cdots P_0 x_0 = \sum_{z \in
\widetilde{\mathcal{M}}(t,s)} \negthickspace z^\top
\widetilde{P}_{t-1} \cdots \widetilde{P}_0 x_{\mathcal{R}}^0 \\
&= \mathbb{P} \big\{x_{\mathcal{R}}(t) \in
\widetilde{\mathcal{M}}(t,s) \big\} = \mathbb{P} \big\{
l_{\mathcal{R}} (\widetilde{\mathcal{U}}(t, x_{\mathcal{R}}(t)),
x_{\mathcal{R}}(t)) = s \big\} .
\end{align*}
Furthermore,
\begin{equation*}
l(\mathcal{U}(0, x_0), x_0) =
l_{\mathcal{R}}(\widetilde{\mathcal{U}}(0, C x_0), C x_0) =
l_{\mathcal{R}}(\widetilde{\mathcal{U}}(0, x_{\mathcal{R}}^0),
x_{\mathcal{R}}^0) .
\end{equation*}
Thus, we get
\begin{equation*}
\mathbb{E} \big[ l(\mathcal{U}(t, x(t)), x(t)) \big] = \mathbb{E}
\big[ l_{\mathcal{R}} (\widetilde{\mathcal{U}}(t,
x_{\mathcal{R}}(t)), x_{\mathcal{R}}(t)) \big]
\end{equation*}
for all $0 \leq t \leq T-1$. A similar argument shows that
$\mathbb{E}\, [g(x(T))] = \mathbb{E} \,
[g_{\mathcal{R}}(x_{\mathcal{R}}(T))]$. The assertion of the lemma
follows from the linearity of expectations.
\end{proof}

\begin{proof} [Proof of Proposition~\ref{prop2}]
Let $J(x_0, \mathcal{U}^\ast)$ be the value of $J$ for the initial
state $x_0$ and the control policy $(t,x) \mapsto \mathcal{U}^\ast
(t,x) = \mathcal{U}_{\mathcal{R}}^\ast (t, Cx)$, and let
$J_{\mathcal{R}}(x_{\mathcal{R}}^0, \mathcal{U}_{\mathcal{R}}^\ast)$
be the value of $J_{\mathcal{R}}$ when the initial state is
$x_{\mathcal{R}}^0 = C x_0$ and the control policy
$(t,x_{\mathcal{R}}) \mapsto \mathcal{U}_{\mathcal{R}}^\ast (t,
x_{\mathcal{R}})$ is applied. By Lemma~\ref{lem2}, we have $J(x_0,
\mathcal{U}^\ast) = J_{\mathcal{R}}(x_{\mathcal{R}}^0,
\mathcal{U}_{\mathcal{R}}^\ast) = J_{\mathcal{R}}^\ast$. Let $(t,x)
\mapsto \overline{\mathcal{U}}(t,x)$ be the optimal control policy
for $\Sigma$ given by Lemma~\ref{lem1}. Define a control policy for
$\Sigma_{\mathcal{R}}$ by $(t,x_{\mathcal{R}}) \mapsto
\overline{\mathcal{U}}_{\mathcal{R}}(t, x_{\mathcal{R}}) =
\overline{\mathcal{U}}(t,x)$, where $x_{\mathcal{R}} = C x$. Then
$\overline{\mathcal{U}}_{\mathcal{R}}$ is well defined since $C x =
C x'$ implies that $\overline{\mathcal{U}}(t, x) =
\overline{\mathcal{U}}(t, x')$. Let $J(x_0, \overline{\mathcal{U}})$
and $J_{\mathcal{R}}(x_{\mathcal{R}}^0,
\overline{\mathcal{U}}_{\mathcal{R}})$ be the corresponding values
of $J$ and $J_{\mathcal{R}}$ respectively. We have
$J_{\mathcal{R}}(x_{\mathcal{R}}^0,
\overline{\mathcal{U}}_{\mathcal{R}}) = J(x_0,
\overline{\mathcal{U}}) = J^\ast$, by Lemma~\ref{lem2}. Since
$\overline{\mathcal{U}}$ minimizes $J$ given the initial state
$x_0$, it follows that $J(x_0, \overline{\mathcal{U}}) \leq J(x_0,
\mathcal{U}^\ast)$, and thus $J^\ast \leq J_{\mathcal{R}}^\ast$. On
the other hand, $J_{\mathcal{R}}(x_{\mathcal{R}}^0,
\mathcal{U}_{\mathcal{R}}^\ast) \leq
J_{\mathcal{R}}(x_{\mathcal{R}}^0,
\overline{\mathcal{U}}_{\mathcal{R}})$ since
$\mathcal{U}_{\mathcal{R}}^\ast$ minimizes $J_{\mathcal{R}}$ for
given $x_{\mathcal{R}}^0$. Thus, $J_{\mathcal{R}}^\ast \leq J^\ast$
and, hence, they are equal. It is also clear that $\mathcal{U}^\ast$
is an optimal control law since $J(x_0, \mathcal{U}^\ast) =
J_{\mathcal{R}}^\ast = J^\ast$.
\end{proof}

\begin{example} \label{eg3}
To give an intuitive example of the proposed equivalence relation
for solving the optimal control problem, consider again the PBN in
Example~\ref{eg1}. Suppose that the functions $l(u, x)$ and $g(x)$
are given by
\begin{align*}
&l(\delta_2^1, \delta_8^1) = \cdots = l(\delta_2^1, \delta_8^4) = 1,
\; l(\delta_2^1, \delta_8^5) = \cdots = l(\delta_2^1, \delta_8^8)
= 2 , \\
&l(\delta_2^2, x) = 3 , \quad x \in \Delta_8 , \\
&g(\delta_8^1) = 1, \quad g(\delta_8^2) = \cdots = g(\delta_8^8) = 2
.
\end{align*}
Then condition (\ref{eq8a}) defines an equivalence relation
$\mathcal{S}$ corresponding to the partition $\{ \{\delta_8^1\}, \,
\{\delta_8^2, \delta_8^3, \delta_8^4\}, \, \{\delta_8^5, \delta_8^6,
\delta_8^7, \delta_8^8\} \}$. Let $\mathcal{R}$ be an equivalence
relation which is contained in $\mathcal{S}$ and satisfies
(\ref{eq4}); for example, let $\mathcal{R}$ be the relation given in
Example~\ref{eg1}. It is easily checked that, for any $x, x'$ in the
same equivalence class of $\mathcal{R}$ and for either $u \in
\Delta_2$, we have $g(x) = g(x')$, $l(u,x) = l(u,x')$, and under the
same input the (one-step) transition probability from $x$ to any of
the four equivalence classes of $\mathcal{R}$ is equal to that from
$x'$ to that class. For instance, if $x = \delta_8^2$, $x' =
\delta_8^3$, and the input $u = \delta_2^1$, then $g(x) = g(x') =
2$, $l(u,x) = l(u,x') = 1$, and the transition probability from $x$
to the equivalence class $\{\delta_8^2, \delta_8^3\}$ or from $x$ to
$\{\delta_8^1\}$ is $0.5$, which is also the transition probability
from $x'$ to $\{\delta_8^2, \delta_8^3\}$ or to $\{\delta_8^1\}$
(cf. Fig.~\ref{fig1}). This means that the states belonging to the
same equivalence class of $\mathcal{R}$ have similar properties in
terms of costs and transitions, and then can be amalgamated to form
a quotient.
\end{example}

To conclude the section, we mention that the controller synthesized
via Proposition~\ref{prop1} or \ref{prop2} has a specific structure
in which all states in the same equivalence class are assigned the
same control action. There is therefore an underlying assumption
when applying the quotient-based method, namely that such a
controller exists for the original network. We do not explicitly
mention this assumption in the statement of Propositions~\ref{prop1}
and \ref{prop2}, since it is automatically implied by the conditions
already stated in the propositions. Indeed, it follows from
Proposition~\ref{prop1} that if a PBN $\Sigma$ is stabilizable, then
so is the quotient system $\Sigma_{\mathcal{R}}$, and by inducing a
stabilizing controller for $\Sigma_{\mathcal{R}}$ back to the
original network, one can derive a feedback law that stabilizes
$\Sigma$, showing for $\Sigma$ the existence of a stabilizing
controller with that specific structure. Similarly, we see from
Lemma~\ref{lem1} that under the conditions of
Proposition~\ref{prop2}, there always exists for $\Sigma$ an optimal
controller having that structure. We should note, however, that
these existence results do not ensure that we are always able to
find a stabilizing (or optimal) controller for a PBN on the basis of
another controller designed from a smaller network, since there may
be situations in which there is no equivalence relation satisfying
the hypotheses of Proposition~\ref{prop1} (or \ref{prop2}) except
for the identity relation, yielding a quotient system the same as
the original. Also, note that in the above discussion we do not
require the equivalence relation to be maximal, although that will
be the case in most applications. In practice, for a given PBN, we
may apply Theorem~\ref{thm2} to find the maximal equivalence
relation $\mathcal{R}$ that satisfies the hypotheses of
Proposition~\ref{prop1} (or \ref{prop2}). Such a maximal
$\mathcal{R}$ always exists: in the extreme case, one has
$\mathcal{R}$ equal to the identity relation, which means that no
other equivalence relations exist that satisfy the proposition's
hypotheses. According to the preceding argument, if the PBN is
stabilizable, then it can be stabilized by a feedback that assigns
the same control to any two states related by $\mathcal{R}$. Also,
there exists an optimal controller where the control actions
corresponding to different states related by $\mathcal{R}$ are the
same.

\section{A Biological Example}

\label{sec5}

The \textit{lac} operon in \textit{Escherichia coli} is the system
responsible for the transport and metabolism of lactose. Although
glucose is the preferred carbon source for \textit{E. coli}, the
\textit{lac} operon allows for the effective digestion of lactose
when glucose is not readily available. A Boolean model for the
\textit{lac} operon in \textit{E. coli} was identified in
\cite{veliz2011}. The model consists of $13$ variables ($1$ mRNA,
$5$ proteins, and $7$ sugars) denoted by $M_{lac}$, $P_{lac}$, $B$,
$C_{ap}$, $R$, $R_m$, $A$, $A_m$, $L$, $L_m$, $L_e$, $L_{em}$ and
$G_e$. The Boolean functions of the model are given in
Table~\ref{tab1}.
\begin{table}
\caption{Boolean Functions for the lac Operon Network
\cite{veliz2011}} \label{tab1} \centering
\renewcommand{\arraystretch}{1.25}
\begin{tabular}{p{2.5cm}p{3.5cm}}
\hline Variable & Boolean Function \\
\hline $M_{lac}$ & $C_{ap} \wedge \neg R \wedge \neg R_m$ \\
$P_{lac}$ & $M_{lac}$ \\
$B$ & $M_{lac}$ \\
$C_{ap}$ & $\neg G_e$ \\
$R$ & $\neg A \wedge \neg A_m$ \\
$R_m$ & $(\neg A
\wedge \neg A_m) \vee R$ \\
$A$ & $B \wedge L$ \\
$A_m$ & $L \vee L_m$ \\
$L$ & $P_{lac} \wedge L_e \wedge \neg G_e$ \\
$L_m$ & $((L_{em} \wedge P_{lac}) \vee L_e ) \wedge \neg G_e$ \\
\hline
\end{tabular}
\end{table}
We assume that the concentration of extracellular lactose (indicated
by $L_e$ and $L_{em}$) can be either low or medium,\footnote{The
variables $L_e$ and $L_{em}$ are combined to indicate the
concentration levels of extracellular lactose: the concentration is
low when $(L_e, L_{em}) = (0,0)$, medium when $(L_e, L_{em}) =
(0,1)$, and high when $(L_e, L_{em}) = (1,1)$. The fourth
possibility, $(L_e, L_{em}) = (1,0)$, is meaningless and not
allowed. See \cite{veliz2011} for more information.} causing the
model to appear random. We then arrive at a PBN consisting of two
BNs. The first constituent BN is determined from Table~\ref{tab1}
when $L_e = L_{em} = 0$, and the second constituent BN is determined
by setting $L_e = 0$ and $L_{em} = 1$. The two constituent BNs are
assumed to be equally likely. The concentration level of
extracellular glucose ($G_e$) acts as the control input. The
algebraic representation of the PBN is as in (\ref{eq2}), with $N =
1024$, $M = 2$, and the selection probabilities given by $p_1 = p_2
= 0.5$. The matrices $F_1, F_2 \in \mathcal{L}^{1024 \times 2048}$
are not presented explicitly due to their sizes.

\textit{1) Stabilization.} When extracellular lactose is low, the
\textit{lac} operon model is known to exhibit two steady states
\cite{veliz2011}, expressed in the canonical vector form as
$\delta_{1024}^{912}$ and $\delta_{1024}^{976}$. Let $\mathcal{M} =
\{ \delta_{1024}^{912} \}$ and let $\mathcal{S}$ be the equivalence
relation produced by the partition $\{ \mathcal{M}, \Delta_{1024} -
\mathcal{M} \}$. Then by following the procedure described in
Section~\ref{sec3}, we obtain a quotient system
$\Sigma_{\mathcal{R}}$ with the transition probability matrix given
by
\begin{align*}
&\widetilde{P} = \big[ \delta_{33}^4 \;\; \delta_{33}^{23} \;\;
\delta_{33}^{20} \;\; \delta_{33}^4 \;\; \delta_{33}^6 \;\;
\delta_{33}^4 \;\; \delta_{33}^{23} \;\; \delta_{33}^{23} \;\;
\delta_{33}^{23} \;\; \delta_{33}^{20} \;\; \delta_{33}^4 \;\;
\delta_{33}^{11} \;\; \delta_{33}^4 \\
&\delta_{33}^6 \;\; \delta_{33}^4 \;\; \delta_{33}^{11} \;\;
\delta_{33}^{32} \;\; \delta_{33}^{32} \;\; \delta_{33}^{29} \;\;
\delta_{33}^{13} \;\; \delta_{33}^{11} \;\; \delta_{33}^{15} \;\;
\delta_{33}^{13} \;\; \delta_{33}^{11} \;\; \delta_{33}^{32} \;\;
\delta_{33}^{32} \;\; \delta_{33}^{32} \\
&\delta_{33}^{29} \;\; \delta_{33}^{13} \;\; \delta_{33}^{11} \;\;
\delta_{33}^{15} \;\; \delta_{33}^{13} \;\; \delta_{33}^{11} \;\;
\delta_{33}^1 \;\; \delta_{33}^{23} \;\; \delta_{33}^{18} \;\;
\delta_{33}^1 \;\; \delta_{33}^6 \;\; \delta_{33}^4 \;\; 0.5
\delta_{33}^{22} + \\
&0.5 \delta_{33}^{23} \;\, 0.5 \delta_{33}^{23} + 0.5
\delta_{33}^{24} \;\, \delta_{33}^{17} \;\, 0.5 \delta_{33}^{18} +
0.5 \delta_{33}^{19} \;\, \delta_{33}^2 \;\, 0.5 \delta_{33}^3 + 0.5
\delta_{33}^9 \\
& 0.5 \delta_{33}^1 + 0.5 \delta_{33}^5 \:\, 0.5 \delta_{33}^6 + 0.5
\delta_{33}^{24} \:\, 0.5 \delta_{33}^4 + 0.5 \delta_{33}^5 \:\, 0.5
\delta_{33}^{11} + 0.5 \delta_{33}^{21} \\
& \delta_{33}^{32} \;\; \delta_{33}^{25} \;\; \delta_{33}^{27} \;\;
\delta_{33}^7 \;\; \delta_{33}^9 \;\; \delta_{33}^{15} \;\;
\delta_{33}^{13} \;\; \delta_{33}^{11} \;\; 0.5 \delta_{33}^{31} +
0.5 \delta_{33}^{32} \;\; 0.5 \delta_{33}^{32} + \\
& 0.5 \delta_{33}^{33} \;\: 0.5 \delta_{33}^{25} + 0.5
\delta_{33}^{26} \;\: 0.5 \delta_{33}^{27} + 0.5 \delta_{33}^{28}
\;\: 0.5 \delta_{33}^{7} + 0.5 \delta_{33}^{8} \;\: 0.5
\delta_{33}^{9} \\
&+ 0.5 \delta_{33}^{10} \;\; 0.5 \delta_{33}^{15} + 0.5
\delta_{33}^{16} \;\; 0.5 \delta_{33}^{13} + 0.5 \delta_{33}^{14}
\;\; 0.5 \delta_{33}^{11} + 0.5 \delta_{33}^{12} \big] .
\end{align*}
Note that the quotient system $\Sigma_{\mathcal{R}}$ has $33$ states
which is about $3\%$ of the number of states of the original PBN.
The matrix $C$ obtained during the procedure (which is of size $33
\times 1024$ and not shown explicitly) satisfies $C
\delta_{1024}^{912} = \delta_{33}^1$. It is easy to see (by the
method of \cite{r.li2014b}) that the quotient system
$\Sigma_{\mathcal{R}}$ can be stabilized to $\delta_{33}^1$ with
probability one via the feedback law $x_{\mathcal{R}} \mapsto K
x_{\mathcal{R}}$, where $K \in \mathcal{L}^{2 \times 33}$ has
$\delta_{2}^2$ as the first and fourth columns and $\delta_{2}^1$ as
its other columns. Proposition~\ref{prop1} then ensures that the
feedback law $x \mapsto \mathcal{U}(x) = KCx$ stabilizes the
original PBN to the state $\delta_{1024}^{912}$, with probability
one. Specifically, this controller is given as: $\mathcal{U}(x) =
\delta_2^2$ if $x\in \{ \delta_{1024}^{784}, \delta_{1024}^{800},
\delta_{1024}^{816}, \delta_{1024}^{848}, \delta_{1024}^{864},
\delta_{1024}^{880}, \delta_{1024}^{896}, \delta_{1024}^{912},
\delta_{1024}^{928}, \delta_{1024}^{944},$ $\delta_{1024}^{976},
\delta_{1024}^{992}, \delta_{1024}^{1008}, \delta_{1024}^{1024}\}$
and $\mathcal{U}(x) = \delta_2^1$ otherwise. A similar argument can
be made for finding a feedback controller that stabilizes the PBN to
the state $\delta_{1024}^{976}$; the details are not repeated here.

\textit{2) Optimal control.} Assume that $T = 10$, $x_0 =
\delta_{1024}^1$, and the functions $l(u,x)$ and $g(x)$ are given by
\begin{align*}
l(\delta_2^1, x) &= 1, \quad l(\delta_2^2, x) = 0 , \quad x \in
\Delta_{1024} , \\
g(\delta_{1024}^1) &= \cdots = g(\delta_{1024}^{128}) = 3 , \quad
g(\delta_{1024}^{129}) = \cdots = g(\delta_{1024}^{1024}) = 6 .
\end{align*}
Here we mention that $\delta_{1024}^1 , \ldots, \delta_{1024}^{128}$
are exactly the states corresponding to the \textit{lac} operon
being ON (cf. \cite{veliz2011}). The above choice of $g(x)$ then
indicates that ON states are more desirable. By proceeding as in
Section~\ref{sec4.2}, one can obtain a quotient system
$\Sigma_{\mathcal{R}}$ with the transition probability matrix given
by
\begin{align*}
&\widetilde{P} = \big[ \delta_{25}^{16} \;\: \delta_{25}^{16} \;\:
\delta_{25}^{16} \;\: \delta_{25}^{16} \;\: \delta_{25}^{16} \;\:
\delta_{25}^{24} \;\: \delta_{25}^{24} \;\: \delta_{25}^{24} \;\:
\delta_{25}^{16} \;\: \delta_{25}^{24} \;\: \delta_{25}^{24} \;\:
\delta_{25}^{24} \;\: \delta_{25}^{24} \\
&\delta_{25}^{16} \;\; \delta_{25}^{16} \;\; \delta_{25}^{16} \;\;
\delta_{25}^{16} \;\; \delta_{25}^{24} \;\; \delta_{25}^{24} \;\;
\delta_{25}^{24} \;\; \delta_{25}^{24} \;\; \delta_{25}^{16} \;\;
\delta_{25}^{16} \;\; \delta_{25}^{16} \;\; \delta_{25}^{16} \;\;
\delta_{25}^7 \;\; \delta_{25}^6 \\
& 0.5 \delta_{25}^7 + 0.5 \delta_{25}^8\;\; 0.5 \delta_{25}^1 + 0.5
\delta_{25}^2 \;\; 0.5 \delta_{25}^9 + 0.5 \delta_{25}^{16} \;\;
\delta_{25}^{24} \;\; \delta_{25}^{18} \;\; \delta_{25}^{20} \\
& \delta_{25}^2 \;\; 0.5 \delta_{25}^{23} +0.5 \delta_{25}^{24} \;\;
0.5 \delta_{25}^{24} + 0.5 \delta_{25}^{25} \;\; 0.5
\delta_{25}^{18} + 0.5 \delta_{25}^{19} \;\; 0.5 \delta_{25}^{20} +
\\
& 0.5 \delta_{25}^{21} \;\; 0.5 \delta_{25}^2 + 0.5 \delta_{25}^3
\;\; 0.5 \delta_{25}^5 + 0.5 \delta_{25}^{16} \;\; \delta_{25}^{16}
\;\; 0.5 \delta_{25}^{4} + 0.5 \delta_{25}^{16} \\
& 0.5 \delta_{25}^{23} + 0.5 \delta_{25}^{24} \;\; 0.5
\delta_{25}^{24} + 0.5 \delta_{25}^{25} \;\; 0.5 \delta_{25}^{18} +
0.5 \delta_{25}^{19} \;\; 0.5 \delta_{25}^{20} + \\
& 0.5 \delta_{25}^{21} \;\; 0.5 \delta_{25}^2 + 0.5 \delta_{25}^3
\;\; 0.5 \delta_{25}^5 + 0.5 \delta_{25}^{16} \;\; \delta_{25}^{16}
\;\; 0.5 \delta_{25}^4 + 0.5 \delta_{25}^{16} \big] .
\end{align*}
Note that the size of $\Sigma_{\mathcal{R}}$ is less than $2.5 \%$
when compared to the original model. The matrix $C$ satisfies $C x_0
= \delta_{25}^{25}$, and the induced functions $l_{\mathcal{R}}$ and
$g_\mathcal{R}$ are defined by
\begin{align*}
&l_{\mathcal{R}}(\delta_2^1, x_{\mathcal{R}}) = 1, \quad
l_{\mathcal{R}}(\delta_2^2, x_{\mathcal{R}}) = 0, \quad
x_{\mathcal{R}} \in \Delta_{25} , \\
&g_{\mathcal{R}}(\delta_{25}^1) = \cdots =
g_{\mathcal{R}}(\delta_{25}^{17}) = 6 , \quad
g_{\mathcal{R}}(\delta_{25}^{18}) = \cdots =
g_{\mathcal{R}}(\delta_{25}^{25}) = 3 .
\end{align*}
It is not hard to see that\footnote{Similarly as in the proof of
Lemma~\ref{lem1}, the optimal control problem for
$\Sigma_{\mathcal{R}}$ can be solved by the following dynamic
programming algorithm, which proceeds backward in time from $t = 10$
to $t = 0$ (see, e.g., \cite{datta2003,bertsekasbook}):
\begin{align*}
&H(10, x_{\mathcal{R}}) = g_{\mathcal{R}}(x_{\mathcal{R}}) , \quad
x_{\mathcal{R}} \in \Delta_{25} , \\
&H(t, x_{\mathcal{R}}) = \min_{u \in \Delta_2}
G(t,x_{\mathcal{R}},u) = \min_{u \in \Delta_2} \bigg\{\sum_{\xi \in
\Delta_{25}} H(t+1, \xi) \xi^{\top} \widetilde{P}(u) x_{\mathcal{R}} \\
&\qquad \qquad \quad + l_{\mathcal{R}}(u,x_{\mathcal{R}}) \bigg\} ,
\quad x_{\mathcal{R}} \in \Delta_{25} , \quad t = 9, 8, \ldots, 0,
\end{align*}
where $\widetilde{P}(u) = \widetilde{P} \ltimes u$ for $u \in
\Delta_2$. The optimal control law is obtained as
$\mathcal{U}_{\mathcal{R}}^\ast (t,x_{\mathcal{R}}) = \arg \min_{u
\in \Delta_2} G(t,x_{\mathcal{R}},u)$, and the optimal cost starting
from the initial state $x_{\mathcal{R}}^0$ is given by $H(0,
x_{\mathcal{R}}^0)$. Clearly, different initial states may have
different optimal values associated with them. For example, here a
direct computation shows that $H(0, \delta_{25}^{25}) = 5.9063$ and
$H(0, \delta_{25}^1) = 6$. Thus the optimal cost for the initial
state $x_{\mathcal{R}}^0 = \delta_{25}^{25}$ is $5.9063$, while that
for the initial state $x_{\mathcal{R}}^0 = \delta_{25}^1$ is $6$.}
the constant control $u = \delta_2^2$ is optimal for
$\Sigma_{\mathcal{R}}$, with the optimal cost $J_{\mathcal{R}}^\ast
= 5.9063$ (to which corresponds $x_{\mathcal{R}}^0 =
\delta_{25}^{25}$). Thus, by virtue of Proposition~\ref{prop2}, this
constant input also solves the optimal control problem for the
original PBN, and the optimal cost corresponding to the initial
state $x_0 = \delta_{1024}^1$ is $J^\ast = J_{\mathcal{R}}^\ast =
5.9063$.

\section{Summary}

\label{sec6}

We considered quotients for PBNs in the exact sense that the notion
is used in the control community. Specifically, we considered a
probabilistic transition system generated by the PBN. The
corresponding quotient transition system then captures the quotient
dynamics of the PBN. We thus proposed a method of constructing a
probabilistic Boolean system that generates the transitions of the
quotient transition system. It is not surprising that the
equivalence relation should satisfy certain constraints so that the
quotient dynamics can indeed be generated from a Boolean system. We
then developed a procedure converging in a finite number of
iterations to a satisfactory equivalence relation. Finally, a
discussion on the use of quotient systems for control design was
given, and an application of the proposed results to stabilization
and optimal control was presented. As a result, it is concluded that
the control problems of the original PBN can be boiled down to those
of the quotient systems. That is, instead of deriving control
polices directly on the original network, which could be
computationally expensive, one can design control polices on the
quotient and subsequently induce the control polices back to the
original PBN.

\section*{Appendix}

\begin{proof}[Proof of Lemma~\ref{lem1a}]
First, note that $\mathcal{Z}_j \supseteq \mathcal{Z}_{j-1}$. In
fact, since $\mathcal{M}^\ast = \mathcal{M}_\iota =
\mathcal{M}_\iota \cap \mathcal{A}(\mathcal{M}_\iota)$, we have
$\mathcal{Z}_0 \subseteq \mathcal{Z}_1$, and if $\mathcal{Z}_{j-1}
\subseteq \mathcal{Z}_j$, then for any $a \in \Delta_N$ such that
$\sum_{x \in \mathcal{Z}_{j-1}} x^\top P(u) a = 1$ for some $u \in
\Delta_M$, we have $\sum_{x \in \mathcal{Z}_j} x^\top P(u) a = 1$,
and thus $\mathcal{Z}_j \subseteq \mathcal{Z}_{j+1}$.

Now, suppose that there exists a control law $\mathcal{U} \colon
\Delta_N \rightarrow \Delta_M$ that stabilizes the PBN to
$\mathcal{M}$ with probability one. We first show that for every $k
\geq 1$,
\begin{equation} \label{eq17}
x_0 \in \Delta_N \;\; \text{and} \;\; \sum_{x \in \mathcal{M}^\ast}
x^\top P_{\mathcal{U}}^k x_0 = 1 \Rightarrow x_0 \in \mathcal{Z}_k .
\end{equation}
We use induction on $k$. The case $k = 1$ is trivial, so we proceed
to the induction step. If $\sum_{x \in \mathcal{M}^\ast} x^\top
P_{\mathcal{U}}^k x_0 = 1$, then since
\begin{equation*}
\sum_{x \in \mathcal{M}^\ast} x^\top P_{\mathcal{U}}^k x_0 = \sum_{b
\in \Delta_N} \bigg(\sum_{x \in \mathcal{M}^\ast}x^\top
P_{\mathcal{U}}^{k-1}b \bigg) b^\top P_{\mathcal{U}}x_0
\end{equation*}
and since $\sum_{b \in \Delta_N} b^\top P_{\mathcal{U}} x_0 = 1$, we
have
\begin{equation*}
b \in \Delta_N \;\; \text{and} \;\; b^\top P_{\mathcal{U}} x_0 > 0
\Rightarrow \sum_{x \in \mathcal{M}^\ast} x^\top
P_{\mathcal{U}}^{k-1} b = 1 ,
\end{equation*}
and so by the induction hypothesis,
\begin{equation*}
b \in \Delta_N \;\; \text{and} \;\; b^\top P_{\mathcal{U}} x_0 > 0
\Rightarrow b \in \mathcal{Z}_{k-1} .
\end{equation*}
Consequently,
\begin{align*}
&\sum_{x \in \mathcal{Z}_{k-1}} x^\top P(\mathcal{U}(x_0)) x_0 =
\sum_{x \in \mathcal{Z}_{k-1}} x^\top P_{\mathcal{U}} x_0 \\
&\qquad \geq \sum_{\{x \colon x^\top P_{\mathcal{U}}x_0 > 0\}}
x^\top P_{\mathcal{U}}x_0 = \sum_{x \in \Delta_N} x^\top
P_{\mathcal{U}}x_0 = 1 .
\end{align*}
This shows that $x_0 \in \mathcal{Z}_k$.

Let $x_0 \in \Delta_N$. Since the feedback $\mathcal{U} \colon
\Delta_N \rightarrow \Delta_M$ stabilizes the PBN to $\mathcal{M}$
with probability one, there is $\tau \geq 0$ such that $\sum_{x \in
\mathcal{M}} x^\top P_{\mathcal{U}}^k x_0 = 1$ for all $k \geq
\tau$. Fix $k \geq \tau$. Since
\begin{equation*}
\sum_{b \in \Delta_N} \bigg(\sum_{x \in \mathcal{M}}x^\top
P_{\mathcal{U}}b \bigg) b^\top P_{\mathcal{U}}^k x_0 = \sum_{x \in
\mathcal{M}}x^\top P_{\mathcal{U}}^{k+1} x_0 = 1
\end{equation*}
and since $\sum_{b \in \Delta_N} b^\top P_{\mathcal{U}}^k x_0 = 1$,
we see that
\begin{equation*}
b \in \Delta_N \;\; \text{and} \;\; b^\top P_{\mathcal{U}}^k x_0 > 0
\Rightarrow \sum_{x \in \mathcal{M}} x^\top P_{\mathcal{U}} b = 1
\Rightarrow b \in \mathcal{A}(\mathcal{M}_0) .
\end{equation*}
Hence,
\begin{align*}
\sum_{x \in \mathcal{A}(\mathcal{M}_0)} x^\top P_{\mathcal{U}}^k x_0
\geq \sum_{\{x \colon x^\top P_{\mathcal{U}}^k x_0 > 0\}} x^\top
P_{\mathcal{U}}^k x_0 &= \sum_{x \in \Delta_N} x^\top
P_{\mathcal{U}}^k x_0 \\
&= 1 ,
\end{align*}
so that
\begin{multline*}
\sum_{x \in \mathcal{M}_1} x^\top P_{\mathcal{U}}^k x_0 = \sum_{x
\in \mathcal{M}_0} x^\top P_{\mathcal{U}}^k x_0 + \sum_{x \in
\mathcal{A}(\mathcal{M}_0)} x^\top P_{\mathcal{U}}^k x_0 \\
- \sum_{x \in \mathcal{M}_0 \cup \mathcal{A}(\mathcal{M}_0)} x^\top
P_{\mathcal{U}}^k x_0 \geq 1 .
\end{multline*}
This implies that $\sum_{x \in \mathcal{M}_1} x^\top
P_{\mathcal{U}}^k x_0 = 1$, for any $k \geq \tau$. In the same way
and by a simple induction argument, we obtain $\sum_{x \in
\mathcal{M}^\ast} x^\top P_{\mathcal{U}}^k x_0 = 1$ for $k \geq
\tau$, and therefore, by (\ref{eq17}) $x_0 \in \mathcal{Z}_k$ for
all $k \geq \tau$. This implies that $\mathcal{Z}_\lambda =
\Delta_N$ for sufficiently large $\lambda$.

Conversely, suppose that $\mathcal{Z}_\lambda = \Delta_N$ for some
$\lambda \geq 1$. Let $\mathcal{Z}_1' = \mathcal{Z}_1$ and
$\mathcal{Z}_j' = \mathcal{Z}_j - \mathcal{Z}_{j-1}$ for $j = 2,3,
\ldots, \lambda$. For every $x \in \Delta_N$, we find a unique
$\mathcal{Z}_j'$ containing $x$ and then pick $u_x \in \Delta_M$
such that $\sum_{b \in \mathcal{Z}_{j-1}}b^\top P(u_x)x = 1$. We
show that the feedback given by $\mathcal{U} \colon x \mapsto u_x$
stabilizes the PBN to $\mathcal{M}$ with probability one. Since
$\mathcal{Z}_\lambda = \Delta_N$ and $\mathcal{M}^\ast \subseteq
\mathcal{M}$, it suffices to show that for $1 \leq j \leq \lambda$,
\begin{equation} \label{eq18}
x_0 \in \mathcal{Z}_j \;\; \text{and} \;\; k \geq j \Rightarrow
\sum_{x \in \mathcal{M}^\ast} x^\top P_{\mathcal{U}}^k x_0 = 1 .
\end{equation}
We use induction on $j$. By the definition of $\mathcal{U}$, we have
$\sum_{x \in \mathcal{M}^\ast} x^\top P_{\mathcal{U}} x_0 = 1$ for
all $x_0 \in \mathcal{Z}_1$. If $k \geq 2$ and if $\sum_{x \in
\mathcal{M}^\ast} x^\top P_{\mathcal{U}}^{k-1} x_0 = 1$ for all $x_0
\in \mathcal{Z}_1$, then for fixed $x_0$ we have
\begin{align*}
\sum_{x \in \mathcal{M}^\ast} x^\top P_{\mathcal{U}}^k x_0 &\geq
\sum_{b \in {\mathcal{M}^\ast}} \bigg(\sum_{x \in
\mathcal{M}^\ast}x^\top P_{\mathcal{U}}b \bigg) b^\top
P_{\mathcal{U}}^{k-1} x_0 \\
&= \sum_{b \in {\mathcal{M}^\ast}} b^\top P_{\mathcal{U}}^{k-1} x_0
= 1 .
\end{align*}
Thus (\ref{eq18}) holds for $j = 1$. To prove the induction step,
assume that $j \geq 2$ and (\ref{eq18}) is true for $j - 1$. Let
$x_0 \in \mathcal{Z}_j$, $k \geq j$, and we show that $\sum_{x \in
\mathcal{M}^\ast} x^\top P_{\mathcal{U}}^k x_0 = 1$. This is clear
if $x_0 \in \mathcal{Z}_{j-1}$, by the induction hypothesis; so
suppose $x_0 \in \mathcal{Z}_j'$. Then, by the definition of
$\mathcal{U}$, we obtain $\sum_{b \in \mathcal{Z}_{j-1}} b^\top
P_{\mathcal{U}} x_0 = 1$. Note that
\begin{equation} \label{eq19}
\sum_{x \in \mathcal{M}^\ast} x^\top P_{\mathcal{U}}^k x_0 \geq
\sum_{b \in {\mathcal{Z}_{j-1}}} \bigg(\sum_{x \in
\mathcal{M}^\ast}x^\top P_{\mathcal{U}}^{k-1} b \bigg) b^\top
P_{\mathcal{U}} x_0 .
\end{equation}
Since by the induction assumption $\sum_{x \in \mathcal{M}^\ast}
x^\top P_{\mathcal{U}}^{k-1} b = 1$ for all $b \in
\mathcal{Z}_{j-1}$, the right-hand side of (\ref{eq19}) is equal to
$\sum_{b \in {\mathcal{Z}_{j-1}}} b^\top P_{\mathcal{U}} x_0 = 1$.
This completes the induction step and hence the proof.
\end{proof}

\end{document}